\documentclass[a4paper]{article}

\usepackage[utf8]{inputenc}
\usepackage[T1]{fontenc}
\usepackage{dsfont}
\usepackage{lmodern}
\usepackage[a4paper]{geometry}
\usepackage{enumitem}
\usepackage{verbatim}
\usepackage{caption}
\usepackage{subcaption}
\usepackage{graphicx}
\usepackage{amsmath}
\usepackage{amsthm}
\usepackage{amsfonts}
\usepackage{tikz}
\usepackage{mathrsfs}
\usepackage{amssymb}
\usepackage{graphicx}
\usepackage{amsmath}
\usepackage{csquotes}
\usepackage{mathrsfs}
\usepackage{framed}
\usepackage[english]{babel}
\usepackage[citecolor=blue,colorlinks=true]{hyperref}
\usepackage[utf8]{inputenc}

\newcommand{\N}{\mathbb{N}}
\newcommand{\R}{\mathbb{R}}
\newcommand{\E}{\mathbb{E}}
\newcommand{\PP}{\mathbb{P}}

\def\cald{{\cal D}}

\def\calf{{\cal F}}

\def\calh{{\cal H}}

\def\call{{\cal L}}

\def\calu{{\cal U}}

\def\call{{\cal L}}

\def\eps{{\varepsilon}}

\newcommand\abs[1]{\left| #1 \right|}
\newcommand\norm[1]{\left|\left| #1 \right|\right|}
\usepackage{dsfont}
\usepackage{enumitem}
\usepackage{verbatim}
\def\<{\left\langle }
\def\>{\right\rangle }

\theoremstyle{plain}
\newtheorem{theorem}{Theorem}[section]
\newtheorem{proposition}[theorem]{Proposition}
\newtheorem{lemma}[theorem]{Lemma}

\theoremstyle{definition}
\newtheorem{definition}[theorem]{Definition}

\newtheorem{assumption}[theorem]{Assumption}
\newtheorem{notation}[theorem]{Notation}

\theoremstyle{remark}
\newtheorem{remark}[theorem]{Remark}

\usepackage{lipsum}
\usepackage{amsfonts}
\usepackage{graphicx}
\usepackage{epstopdf}
\usepackage{algorithmic}

\usepackage{amsopn}

\title{A Stochastic Model of Economic Growth in Time-Space}

\author{Fausto Gozzi$^{\dagger}$ \and
Marta Leocata\footnote{Gozzi (\texttt{fgozzi@luiss.it}) and Leocata (\texttt{mleocata@luiss.it}) are at the
Deparment of Economics and Finance,
Luiss University, Roma, Italy.}
}

\begin{document}

\maketitle

\begin{abstract}
We deal with an infinite horizon, infinite dimensional stochastic optimal control problem arising in the study of economic growth in time-space. Such a problem has been the object of various papers in deterministic cases when the possible presence of stochastic disturbances is ignored
(see, e.g., [P.~Brito, {\em The Dynamics of Growth and Distribution in a Spatially
  Heterogeneous World},  working paper 2004/14, ISEG-Lisbon School of Economics and Management, University of Lisbon, 2004], [R.~Boucekkine, C.~Camacho, and G.~Fabbri, {\em J.~Econom.~Theory}, 148 (2013),
  pp.~2719--2736], [G.~Fabbri, {\em J.~Econom.~Theory}, 162 (2016),
  pp.~114--136], and [R.~Boucekkine, G.~Fabbri, S.~Federico, and F.~Gozzi, {\em  J.~Econom.~Geography}, 19 (2019), pp.~1287--1318]).
Here we propose and solve a stochastic generalization of such models where the stochastic term, in line with the standard stochastic economic growth models
(see, e.g., the books [A.~G.~Malliaris and W.~A.~Brock, {\em Stochastic Methods in Economics and
  Finance}, Advanced Textbooks in Economics~17, North Holland, 1982, Chapter~3] and [H.~Morimoto, {\em Stochastic Control and Mathematical Modeling:~Applications in Economics},
  Cambridge Books,  2010, Chapter~9]),
is a multiplicative one, driven by a cylindrical Wiener process.
The problem is studied using the dynamic programming approach.
We find an explicit solution of the associated HJB equation, use a verification type result to prove that such a solution is the value function, and find the optimal feedback strategies.
Finally, we use this result to study the asymptotic behavior of the optimal trajectories.
\end{abstract}

\textbf{Key words}:
Stochastical optimal control problems in infinite dimension with state constraints;
Dynamic programming;
Second order Hamilton-Jacobi-Bellman equations in infinite dimension;
Verification theorems and optimal feedback controls;
Spatial AK model of economic growth;
Invariant measure.

\textbf{AMS classification}:
93E20, 49L20, 35R15, 91G80, 49K27, 93C20, 93E03, 60H15, 60H30.

%

\section{Introduction}
Economic growth problems play a central role in modern economic theory. Their modeling, in a wide variety of cases (starting from \cite{Ramsey29}), uses as a central tool the optimal control theory.
In recent decades various papers have studied
economic growth problems where the state/control variables such as
capital and consumption depend not only on time $t$, but also on space $x$. We identify this research area as one of ``spatial growth''; see, e.g., the papers
\cite{Brito}, \cite{BCZ09}, \cite{BoucekkineCamachoFabbriJET},
\cite{FabbriJET}, \cite{BFFGJOEG19}, \cite{XY16}, \cite{BrockX08}, \cite{BrockX09}.

On the other hand, to take account of the role of uncertainty, several papers in the economics/mathematics literature have considered economic growth problems where the
state/control variables are stochastic processes and the state equation is a stochastic differential equation (SDE).
See, e.g.,  \cite[Chapter 17]{Acemoglubook}
for discrete time and the books \cite[Chapter 3]{MalliarisBrock81}, \cite[Chapter 9]
{Morimotobook}
and the papers
\cite{Bourguignon74}, \cite{Bismut75},
\cite{Merton75}, \cite{Obstfeld94},
\cite{Turnovsky93}, \cite{Turnovsky95}, \cite{Turnovsky96} for continuous time.

The present paper is the first that tries to put together these two approaches. The resulting problem is an infinite dimensional stochastic control problem, i.e., a problem where the state equation is a stochastic PDE driven by a cylindrical Wiener process in a Hilbert space.
The theory of such problems is quite recent (see the book \cite{FabbriGozziSwiech17} for an account of the theory) and incomplete, particularly in cases like the current one, which involves state constraints and unboundedness of the data such as running objective and control set.

To be more precise we consider a stochastic multiplicative perturbation of the deterministic model of \cite{BFFGJOEG19}; thus the horizon is infinite, the state equation is a bilinear controlled stochastic parabolic equation, and the objective functional is homogeneous.

Despite the additional difficulties created by the presence of the stochastic multiplicative term, we are able to completely solve the problem in the case when state constraints require the capital to live in a half space.
We find an explicit solution of the associated HJB equation
(Theorem \ref{th:HJB_particular solution}), use a verification type result to prove that such a solution is the value function, and find the optimal feedback strategies (Theorem \ref{teo:verification_teo}).
Moreover, we use this result to study the asymptotic behavior of the optimal trajectories, finding two results on the convergence of the optimal state (capital) path (Theorems \ref{teo:convtoK_infty} and \ref{teo:convto0}). 

From the technical point of view we observe that the guess for the value function is completely analogous to that of the corresponding deterministic case, based on the homogeneity of the problem. On the other hand, the method for finding the optimal feedback controls is different due to the difficulties brought by the stochastic term; in particular we need to use a different approach for different values of the elasticity parameter $\sigma$ (see section~\ref{sec:4}).
Last, but not least, the asymptotic analysis for the optimal state capital is completely new, and new techniques are required to study the convergence of the stochastic terms. 
More details on this are given in the body of the paper.

The plan of the paper is the following.
In section~\ref{sec:2} we briefly present the problem and the main assumptions, and we
discuss the well posedness of the state equation (Theorem \ref{teo:existence}, proved in Appendix~\ref{app:A}).
In section~\ref{section:1} we find the explicit solution to the HJB equation
(Theorem \ref{th:HJB_particular solution}).
Section~\ref{sec:4} is devoted to finding the optimal controls through a verification theorem (Theorem \ref{teo:verification_teo}):
it is divided into two subsections treating the cases of
positive and negative power utility, respectively.
Section~\ref{sec:5} provides the two results on the limit at infinity, in probability and in law, of the state process (Theorems \ref{teo:convto0}
and \ref{teo:convtoK_infty}), comparing the results with those of the deterministic model.

\section{The optimal stochastic control problem}\label{sec:2}
We formulate here the stochastic version of the AK model
of \cite{BFFGJOEG19} that we study in this paper.

\subsection{The deterministic problem}

To make clear the novelties of our problem, we start by recalling such a deterministic model (in the slightly more general form studied in \cite[section 5]{CFG}).

The space variable $x$ belongs to $S^1$.
The state is the capital $k$, while the control is the
consumption $c$.
In \cite[section 5]{CFG}, the state
equation for the capital is
\begin{equation}\label{eq:K}
\begin{cases}
dk(t,x)=\frac{\partial^2}{\partial x^2}k(t,x)+A(x)k(t,x)-c(t,x) N(x),\qquad
(t,x)\in \R_+\times S^1,\\
k(0,x)=k_0(x), \qquad x \in S^1,
\end{cases}
\end{equation}
where $K(t,x)$ (respectively, $c(t,x)$) is the capital (the per capita consumption) at time $t$ in the location $x\in S^1$, $A(x)$ is the exogenous location-dependent technological level, and $N(x)$ is the density of population ($A,N$ are strictly positive functions in $L^\infty(S^1)$).
We can reformulate this equation in an abstract formulation
in the Hilbert space $\calh:= L^2(S^1)$; we use $K(t):=k(t,\cdot)$ and, with a little abuse of notation, $c(t):=c(t,\cdot)$:
\begin{equation}\label{eq:K_abstract}
\begin{cases}
K'(t)=\mathcal{L}K(t)-c(t)N
,\,\, t\ge 0,\\
K(0)=K_0\in \calh=L^2(S^1),
\end{cases}
\end{equation}
where $\mathcal{L}:\mathcal{D}(\mathcal{L}):\calh\to \calh$ is the linear unbounded operator\footnote{This operator can also be seen as an operator on $[0,2\pi]$
with periodic boundary conditions (in this equivalent setting
the domain is $H^2([0,2\pi])$ with boundary conditions $k(0)=k(2\pi)$ and $k'(0)=k'(2\pi)$).}
$$
\mathcal{D}(\mathcal{L})=H^2(S^1)\subset \calh,
\qquad
(\mathcal{L}k)(x)=k''(x)+A(x)k(x).
$$
As recalled from \cite[section 5]{CFG}, the operator $\call$ generates an analytic semigroup $\{e^{t\call}\}_{t\ge 0}$ and is diagonal with respect to an orthonormal complete system $\{e_j\}_{j \in \N}$.
The eigenvalues $(\lambda_j)_{j\in \N}$ are taken decreasing in $j$, and $e_0$ is continuous and strictly positive.

For every $K_0\in \calh$ and $C\in L^1_{loc}(\R^+,\calh)$ the state equation \eqref{eq:K} admits a unique mild solution (which we call $K^{K_0,c}$) given by
$$
K^{K_0,c}(t)=
e^{t\mathcal{L}}K_0+\int_{0}^{t}e^{(t-s)\mathcal{L}}c(t)N ds.
$$
The reward functional is\footnote{Here and later the subscript $D$ stands for ``deterministic,'' as opposed to the subscript $S$, used later (subsection \ref{SSE:STOCHPROBLEM}), which stands for ``stochastic.''}
\[
J_D(c)=
\int_0^\infty
e^{-\rho t}\mathcal{U}(c(t))dt,
\]
where
\[
\mathcal{U}(c(t))=
\int_0^{2\pi}\frac{c(t,x)^{1-\sigma}}{1-\sigma}f(x)dx,
\]
with $\rho>0$, $\sigma\in(0,1)\cup (1,\infty)$, $f\in L^\infty(S^1,\R^+)$.
The set of admissible controls is
\[
\mathcal{A}_D(K_0)=\{C\in L^1_{loc}(\R^+,\calh)|\,
C(t)(x)\ge 0, \; K^{K_0,c}(t)\geq 0\ \forall t\geq 0
\},
\]
and the value function is
\[V_D(K_0)=\sup_{c\in \mathcal{A}_D(K_0)}J_D(c).\]
However, since the positivity constraint on the capital is very difficult to treat, in \cite{BFFGJOEG19,CFG}, as a first step, one considers the weaker constraint $\<K^{K_0,c}(t),e_0\>_H \ge 0$, where $e_0$ is the (strictly positive) eigenvector of norm $1$ in $\calh$, associated to the highest eigenvalue $\lambda_0$ of $\mathcal{L}$.\footnote{We recall (see \cite[section 5]{CFG})
that $\call$ is a diagonal operator with respect to an orthonormal basis $\{e_n\}_{n\in\N}$ of $\calh$. We call $\{\lambda_n\}_{n\in\N}$ the sequence of eigenvalues which are taken in decreasing order.}
Hence the new set of admissible controls is
\[
\mathcal{A}_{D,e_0}(K_0)=\{C\in L^1_{loc}(\R^+,\calh)|\,
C(t)(x)\ge 0, \; \<K^{K_0,c}(t),e_0\>\geq 0\ \forall t\geq 0
\},
\]
and the corresponding value function is
\[V_{D,e_0}(K_0)=\sup_{c\in \mathcal{A}_{D,e_0}(K_0)}J_D(c).\]
Clearly, for all $K_0\ge 0$ we have
$$
\mathcal{A}_{D}(K_0)\subseteq \mathcal{A}_{D,e_0}(K_0)
\quad \Longrightarrow \quad
V_{D}(K_0) \le
V_{D,e_0}(K_0).
$$
In \cite{BFFGJOEG19} the problem with the weaker constraint is completely solved by finding an explicit solution of the HJB equation and the optimal feedback strategies. In \cite[section 5]{CFG} the connection with the initial, more difficult, problem with positivity constraint is studied by setting the problem in the space of continuous functions, which is harder to treat as it is a Banach nonreflexive space.

\subsection{The stochastic perturbation of the state equation}

Our goal here is to consider a stochastic perturbation of the state equation \eqref{eq:K_abstract}, consequently
changing the objective functional and the set of admissible controls.

The stochastic perturbation is written, in the abstract equation
\eqref{eq:K_abstract}, as
$$B(K(t))dW(t).$$
Here we have the following:
\begin{itemize}
  \item
$W$ is a cylindrical Wiener process
in the space $\calh$,
with a stochastic basis $(\Omega, \calf, (\calf_t)_{t\ge 0},\PP)$,
where the filtration $(\calf_t)_{t\ge 0}$ is the one generated by $W$ augmented with the $\PP$-null sets.
  \item
The operator $B:\calh\to \mathbf{B}(\calh)$ is a linear bounded operator.
In what follows we will denote by
$\mathbf{B}(\calh)$ the space of bounded linear operators
$\calh\to \calh$ and by
$\mathbf{HS}(\calh)$ the space of Hilbert--Schmidt operators from $\calh$ in $\calh$.
  \end{itemize}
The choice of $B$ linear is in line with the already quoted stochastic growth models (see, e.g., \cite[Chapter 3]{MalliarisBrock81}, \cite[Chapter 9]
{Morimotobook}), where the state equation is one-dimensional and the noise term is of the type $\sigma k(t) d\beta(t)$ for some $\sigma\in \R$ and one-dimensional Wiener process $\beta(\cdot)$.

Before proceeding we provide a remark on a straightforward multidimensional generalization of the
one-dimensional multiplicative noise term
$\sigma k(t) d\beta(t)$.

\begin{remark}
\label{rm:GBM}
The typical $d$-dimensional generalization of
the one-dimensional term $\sigma k(t) d\beta(t)$
is $(\sigma_i k_i(t)d\beta_i(t))_{i=1,\dots,d}$,
where $\beta(\cdot)$ is a Wiener process in $\R^d$,
$\sigma \in \R^{d}$, $k(t)\in \R^d$.
For example, in the case of multidimensional
geometric Brownian motion the equations are
\[
dk_{i}(t)=\mu_ik_i(t)dt+
\sigma_i k_i(t) d\beta_i(t),\,\,i=1,\dots,d,\]
with $\mu_i,\sigma_i\in \R^+$. We rewrite this in  vectorial notation, in terms of the diagonal matrices $\Sigma(k)$, $M$, as
\begin{equation}
\Sigma(k)=\begin{pmatrix}
\sigma_1 k_1 & &\\
&\ddots &\\
&&\sigma_d k_d \\
\end{pmatrix}, \qquad\qquad
M=\begin{pmatrix}
\mu_1& &\\
&\ddots &\\
&&\mu_d \\
\end{pmatrix},
\end{equation}
and, in terms of the vectors $\beta(t)=\left(\beta_1(t),\dots,\beta_d(t)\right)$, $k(t)=\left(k_1(t),\dots,k_d(t)\right)$ and $(e_i)_{1=1,\dots,d}$ (the orthonormal canonical basis of $\R^d$), as

\[
dk(t)
=M k(t)dt+\Sigma(k(t)) d\beta(t)=M k(t)dt+\sum_i\sigma_i\langle k(t),e_i \rangle e_i d\beta_i(t).
\]
It is quite clear that in the above model the canonical basis can be replaced by any orthonormal basis in $\R^d$.
\end{remark}

Following along the lines of the above remark,
a reasonable choice of our operator $B$
is to choose it diagonal with respect to the orthonormal system
$\{e_j\}_{j\in\N}$ of eigenvectors of the operator $\mathcal{L}$.
Moreover, for the moment we allow the analogous coefficients $\sigma_i$ of Remark \ref{rm:GBM} to be possibly state-dependent.
Hence we define the operator
\begin{align*}
 B:\mathcal{H}&\to \mathcal{L}(\mathcal{H};\mathcal{H})\\
 K&\mapsto B(K)
\end{align*}
on the orthonormal basis $\{e_j\}_{j\in\N}$ as
\begin{equation}\label{eq:B_ej}
B(K)e_j=\alpha_j(K)\langle K,e_j \rangle e_j,
\end{equation}
where $\alpha_j:\mathcal{H}\to\R$ is a given function.
Consequently, for each $u\in \mathcal{H}$,
\[
B(K)u=\sum_j\langle u,e_j\rangle\alpha_j(K)\langle K,e_j \rangle e_j.
\]
Then writing, formally, $W(t)=\sum_j \beta_j(t)e_j$, we have
\[B(K)W(t)=\sum_j \beta_j(t)\alpha_j(K)\langle K,e_j \rangle e_j .\]
We will use, sometimes separately,
the following assumptions on the family $\{\alpha_j\}_j$.

\begin{assumption}
\label{ass:state}
\begin{enumerate}
\item[]
\item[{\rm (i)}] \label{condition_0}
$\alpha_j$ is measurable and bounded, and
$\alpha_\infty=\sup_j\norm{\alpha_j}^2_\infty<\infty$.
\item[{\rm (ii)}] \label{condition_2} $\alpha_j$ is Lipschitz-continuous for each $j$ and uniformly in $j\in \N$.
\item[{\rm (iii)}]\label{condition_4}  The first component $\alpha_0$ does not depend on the state variable $K$.\footnote{This condition is needed  to derive an explicit solution for the associated optimal control problem.}
\end{enumerate}
\end{assumption}


The resulting state equation is then
\begin{equation}\label{eq:K_abstractstoch}
\begin{cases}
dK(t)=[\mathcal{L}K(t)-C(t)N]dt + B(K(t))dW(t)
,\,\, t\ge 0,\\
K(0)=K_0\in \calh=L^2(S^1).
\end{cases}
\end{equation}

The following theorem establishes the well posedness of
\eqref{eq:K_abstractstoch} and is proved in Appendix \ref{appSE}.

\begin{theorem}\label{teo:existence}
Let Assumptions \eqref{ass:state}{\rm (i) and (ii)} hold. Given the initial condition $K_0\in \mathcal{H}$, equation \eqref{eq:K_abstract} admits a unique mild solution, and for every $p>0$ there exists $C_p$ such that
\[\sup_{t\in[0,T]}\E\left[\abs{K^{K_0,c}(t)}^p\right]\leq C_p\left(1+\norm{K_0}^p_{\mathcal{H}}\right).\]
\end{theorem}

\subsection{Our stochastic control problem}
\label{SSE:STOCHPROBLEM}

Similarly to the deterministic problem of~\cite{BFFGJOEG19,CFG},
we assume that the policy maker operates in order  to maximize the functional
\[J_S(c)=\E\left[\int_0^\infty e^{-\rho t}\mathcal{U}(c(t))dt\right]\]
over all $c\in \mathcal{A}_S(K_0)$, where
\[
\mathcal{A}_S(K_0)=\{c\in L^2(\Omega, L^1_{loc}(\R^+,\calh))|\, c(t) \ge 0,\, K^{K_0,c}(t)\geq 0\ \forall t\geq 0,\,\,\text{a.s.}\},
\]
and, as in the deterministic case,
\[\mathcal{U}(c(t))=
\int_0^{2\pi}\frac{c(t,x)^{1-\sigma}}{1-\sigma}f(x)dx,\]
with $\rho>0$, $\sigma\in(0,1)\cup (1,\infty)$, $f\in C(S^1,\R^+)$.
We call $\mathbf{(P)}$ the problem
\begin{equation}\label{P}\tag{{\bf P}}
    \text{maximize } J_S(c) \text{ over } c\in  \mathcal{A}_S(K_0),
\end{equation}
and we define the value function
\[V_S(K_0)=\sup_{c\in \mathcal{A}_S(K_0)}J_S(c).\]
As in the deterministic case (see  \cite{CFG}), we see that the problem with the positivity constraints (i.e., the admissible set $\mathcal{A}_S(K_0)$) is not explicitly solvable and  is hard to treat since the positive cone in the space $\calh$ has empty interior.
Hence, as a first step, we consider the problem $\mathbf{(P_0)}$ with the weaker constraint $\<K,e_0\>\ge 0$, i.e., with the set of admissible controls given by
\[
\mathcal{A}_{S,e_0}(K_0)=\{c\in L^2(\Omega, L^1_{loc}(\R^+,\calh))|\, c(t) \ge 0,\,\<K^{K_0,c}(t),e_0\>\geq 0\ \forall t\geq 0,\,\,\text{a.s.}\}.
\]
The corresponding value function is
\[V_{S,e_0}(K_0)=\sup_{c\in \mathcal{A}_{S,e_0}(K_0)}J_S(c).\]
As in the deterministic case,
$$
\mathcal{A}_{S}(K_0)\subseteq \mathcal{A}_{S,e_0}(K_0)
\quad \Longrightarrow \quad
V_{S}(K_0) \le
V_{S,e_0}(K_0).
$$
It is clear that if an optimal strategy for problem $\mathbf{(P_0)}$
(with initial datum $K_0$) lies in $\mathcal{A}_{S}(K_0)$,
then it must be optimal also for problem $\mathbf{(P)}$; however, in general this is a very difficult task.
In what follows we will solve completely the problem $\mathbf{(P_0)}$, providing also some results on the asymptotic behavior of optimal trajectories. The results presented in the paper will provide a basis for the study of the original problem, whose full treatment will be the object of a subsequent work.

We will need the following assumptions: the first is the usual one which guarantees the finiteness of the value function
(see, e.g., \cite{BFFGJOEG19,CFG} or, in different contexts,
\cite{FGSET,DiGiacintoFedericoGozzi11});
the second assumption is coherent with the economic problem and allows one to avoid the addition of heavy integrability conditions like those of \cite[Assumption 3.3]{CFG}. 

\begin{assumption}
\label{ass:cost}
\item[]
\begin{itemize}
\item[{\rm (i)}] \label{condition_5}
We have
\[\rho>\lambda_0(1-\sigma)-\frac{1}{2}\sigma(1-\sigma)\alpha^2_0.\]
\item[{\rm (ii)}] \label{condition_6}
The functions $A$, $N$, and $f$ belong to $L^\infty (S^1;\R_+)$, and
for some $\eps >0$ we have $N(x)\ge \eps$ for all $x \in S^1$. (Note that the first eigenvector $e_0$ is continuous and strictly positive.)
\end{itemize}
\end{assumption}

\begin{notation}
\label{not:H+}
The following spaces will be used in subsequent steps:
\[
\calh^{+}=\{f\in \mathcal{H}:f\geq 0\}, \qquad
\calh^{++}=\{f\in \mathcal{H}:f> 0\},
\]
\[\calh^+_{e_0}=\{f\in \mathcal{H}:\langle f,e_0\rangle\geq 0\}, \qquad
\calh^{++}_{e_0}=\{f\in \mathcal{H}:\langle f,e_0\rangle> 0\}
={\rm Int} \calh^+_{e_0}
.\]
With this notation we have
\[{\mathcal{A}}_{S,e_0}(K_0)=\{c\in L^p(\Omega, L^1_{loc}(\R^+,\calh^+))\,\,\forall p\geq 0|\, K^{K_0,c}(t)\in \calh^+_{e_0}\,\,\, \forall t\geq 0,\,\,a.s.\}.
\]
\end{notation}

\begin{remark}
Notice that Assumption \ref{ass:cost}(ii) implies the following integrability conditions (see Assumption 3.3 in \cite{CFG}):
\[
\int_{0}^{2\pi}\left(f(x)\right)^{1/\sigma}\left(N(x)e_0(x)
\right)^{\frac{\sigma-1}{\sigma}}dx<+\infty,\quad \int_{0}^{2\pi}\left(\frac{f(x)}{N(x)e_0(x)}
\right)^{2/\sigma}dx<+\infty. 
\] 
\end{remark}

\begin{remark}
In the current paper, the term $f$ stands for a spatial weight of the utility from consumption considered by the social planner. Since the consumption is made by the population, it is reasonable
to take $f(x)=f_0(N(x))$ for some increasing function $f_0:\R_+\to \R_+$ such that $f_0(0)=0$. A typical case is $f_0(N)=N^\beta$ for some $\beta \ge 0$ (see, e.g., \cite[section 5]{CFG}).
In the case of \cite{BFFGJOEG19} we have $f=N$ (i.e., $\beta=1$), and the planner sums up the utilities of all the individuals in location $x$ and time $t$. On the other hand, when $f_0$ is constant (i.e., $\beta=0$), the planner does not consider the individuals, but only the space average of the consumption. 
\end{remark}

\section{Explicit solution for the HJB equation}
\label{section:1}
The HJB equation associated to the stochastic optimal control problem ${\bf(P_0)}$ is the following:
\begin{align}\label{eq:HJB}
\rho v(K)=\langle K,\mathcal{L}D v(K)\rangle+\sup_{c\in H^+}\{\mathcal{U}(c)-\langle cN,D v(K)\rangle\}+\frac{1}{2}\text{Tr}\left[B(K)B(K)^*D^2v(K)\right]&,\\
\text{with}\,\,\,K\in \calh^+_{e_0}&.\nonumber 
\end{align}
To the best of our knowledge there are no results on viscosity solutions for 
(\ref{eq:HJB}). For some works on the subject see \cite[Chapter 3]{fabbri2017stochastic} and \cite{Cosso2018},
\cite{Rosestolato2016}, \cite{Ren2020}. 

Note that here $v$ is the unknown and
\begin{itemize}
  \item
the term
$\langle \mathcal{L}K,D v(K)\rangle$, which is well defined only for $K\in \cald(\mathcal{L})$,
is replaced here by $\langle K,\mathcal{L}D v(K)\rangle$ which, if $Dv\in\cald(\call)$, makes sense for all $K\in \calh$;
  \item
the map
\begin{equation}\label{eq:HCV}
H_{CV}:\calh\times \calh_+\to \R, \qquad H_{CV}(\alpha;c)=
\mathcal{U}(c)-\langle cN,D \alpha \rangle
\end{equation}
 is called the \textit{current value Hamiltonian};
  \item
the map
\begin{equation}\label{eq:HMAX}
H_{MAX}:\calh \to \R\cup\{+\infty\}, \qquad
H_{MAX}(\alpha)=\sup_{c\in H^+}H_{CV}(\alpha;c)
\end{equation}
is called the
\textit{maximum value Hamiltonian} and is equal to
$\mathcal{U}^*(\alpha)$, the Legendre transform of $\calu$.
\end{itemize}
We start by guessing an explicit solution of \eqref{eq:HJB}.

\begin{definition}
\label{df:sol}
We say that a function $v$ is a classical solution of \eqref{eq:HJB} over $\calh^{++}_{e_0}$ if, in every point of
$\calh^{++}_{e_0}$,
\begin{itemize}
  \item
$v$ admits first and second continuous Fr\'echet derivatives;
  \item
$\mathcal{L}D v$ and $\text{Tr}\left[B(\cdot)B(\cdot)^*D^2v\right]$
are continuous;
  \item
\eqref{eq:HJB} is satisfied.
\end{itemize}
\end{definition}

\begin{theorem}\label{th:HJB_particular solution}
Let Assumptions {\rm \ref{ass:state}(i)} and {\rm \ref{ass:cost}}  hold.
The function
\begin{equation}\label{eq:HJB_particular solution}
w:\calh^{+}_{e_0} \to \R\cup \{-\infty\},
\qquad
w(K)=\gamma\frac{\langle K,e_0 \rangle^{1-\sigma}}{1-\sigma}
\end{equation}
with\footnote{Note that formula \eqref{eq:gamma} is well defined thanks to Assumption \ref{ass:cost}.}
\begin{equation}\label{eq:gamma}
\gamma
=\left(\frac{\sigma}{\rho-\lambda_0(1-\sigma)
+\frac{1}{2}\alpha^2_0\sigma(1-\sigma)} \right)^\sigma\cdot\left(\int_0^{2\pi}
(N(x)e_0(x))^{-\frac{1-\sigma}{\sigma}}
f(x)^{\frac{1}{\sigma}}dx\right)^\sigma
\end{equation}
is a classical solution of \eqref{eq:HJB} over $\calh^{++}_{e_0}$.
\end{theorem}

\begin{proof}
Let $\alpha \in C(S^1,(0,+\infty))$. Then
it is easy to check that the supremum over $c\in \calh^+$
in \eqref{eq:HMAX} is realized by $\alpha^{-1/\sigma}$; hence
\begin{equation}\label{eq:sup}
H_{MAX}(\alpha)=\mathcal{U}^*(\alpha)=\int_0^{2\pi}\frac{\sigma}
{1-\sigma}(N(x)\alpha)^{-\frac{1-\sigma}{\sigma}}
f(x)^{\frac{1}{\sigma}}dx.\end{equation}
Now, to prove that $w(K)=\gamma\frac{\langle K,e_0 \rangle^{1-\sigma}}{1-\sigma}$ is a classical solution of \eqref{eq:HJB}, we compute its derivatives, show that they satisfy the regularity required in Definition \ref{df:sol}, and  plug them into the equation, searching for the value of $\gamma$ which satisfies the identity.
The first and  second Fr\'echet derivatives are given as
\begin{equation}\label{eq:DV}
Dw(K)=\gamma\langle K,e_0\rangle^{-\sigma}e_0,\quad D^2w(K)=-\gamma\sigma\langle K,e_0\rangle^{-\sigma-1}\langle \cdot,e_0\rangle e_0.
\end{equation}
It is immediate to see that
$\call Dw(K)=\gamma\langle K,e_0\rangle^{-\sigma}\call e_0$
is well defined and continuous in $H^{++}_{e_0}$.
Moreover, since $B(K)B(K)^*D^2w(K)$ is a self-adjoint operator, we have
\begin{multline}\label{eq:trace}
        \text{Tr}\left[B(K)B(K)^*D^2w(K)\right]
        =\sum_j\langle B(K)B(K)^*D^2w(K)[e_j],e_j \rangle\\
      =\sum_j-\sigma\gamma\langle K,e_0\rangle^{-\sigma-1}\langle e_j,e_0\rangle\alpha_j(K)^2\langle K,e_0 \rangle^2\langle e_0,e_j \rangle
      =-\sigma\gamma\langle K,e_0\rangle^{-\sigma-1}\alpha_0^2\langle K,e_0 \rangle^2, 
\end{multline}
which is also well defined and continuous in $\calh^{++}_{e_0}$.
Plugging \eqref{eq:sup}, \eqref{eq:DV}, and \eqref{eq:trace} into  \eqref{eq:HJB}, we get
\begin{multline*}
\rho\gamma\frac{\langle K, e_0 \rangle^{1-\sigma}}{1-\sigma}=\langle K, \gamma\langle K,e_0\rangle^{-\sigma} \mathcal{L}e_0\rangle
+\int_0^{2\pi}\frac{\sigma}{1-\sigma}(\gamma N(x)\langle K,e_0\rangle^{-\sigma}e_0)^{-\frac{1-\sigma}{\sigma}}
f(x)^{\frac{1}{\sigma}}dx
\\
-\frac{1}{2}\sigma\gamma\langle K,e_0\rangle^{-\sigma-1}\alpha_0^2\langle K,e_0 \rangle^2.
\end{multline*}
We now use that $\mathcal{L}e_0=\lambda_0e_0$ and divide by
$\langle K,e_0 \rangle^{1-\sigma}>0$, finding
%
\begin{align*}
\frac{\rho}{1-\sigma}\gamma&=\lambda_0\gamma+\frac{\sigma}{1-\sigma}\gamma^{-\frac{1-\sigma}{\sigma}}\int_0^{2\pi}(N(x)e_0(x))^{-\frac{1-\sigma}{\sigma}}f(x)^{\frac{1}{\sigma}}dx
-\frac{1}{2}\sigma\gamma \alpha_0^2.
\end{align*}
Hence
\[
\frac{\sigma}{1-\sigma}\gamma^{-\frac{1}{\sigma}}\int_0^{2\pi}(N(x)e_0(x))^{-\frac{1-\sigma}{\sigma}}f(x)^{\frac{1}{\sigma}}dx
=\frac{\rho}{1-\sigma}-\lambda_0+\frac{1}{2}\sigma \alpha_0^2,\
\]
\[
\gamma^{-\frac{1}{\sigma}}
=\left(\frac{\rho}{\sigma}-\lambda_0\frac{1-\sigma}{\sigma}+\frac{1-\sigma}{2}\alpha_0^2\right)\left(\int_0^{2\pi}(N(x)e_0(x))^{-\frac{1-\sigma}{\sigma}}f(x)^{\frac{1}{\sigma}}dx\right)^{-1},
\]
from which \eqref{eq:gamma} follows.
\end{proof}

\begin{remark}
\label{rm:boundarycond}
Note that the function $w$ is also defined on the boundary of the open set $\calh^{++}_{e_0}$. On this boundary $w$ satisfies the boundary conditions
$$
w(K)=0 \quad \hbox{if $\sigma \in (0,1)$};
\qquad
\lim_{K\to \partial \calh^{++}_{e_0}} w(K)=-\infty \quad
\hbox{if $\sigma \in (1,+\infty)$}.
$$
We will use this fact in what follows. 
\end{remark}

\begin{remark}
It is worth doing a comparison with the deterministic optimal problem. We recall that in the deterministic setting the solution of the HJB, $w_{det}(K)$, differs from $w(K)$ just for the value of $\gamma$, which is \[\gamma_{det}
=\left(\frac{\sigma}{\rho-\lambda_0(1-\sigma)
} \right)^\sigma\cdot\left(\int_0^{2\pi}
(N(x)e_0(x))^{-\frac{1-\sigma}{\sigma}}
f(x)^{\frac{1}{\sigma}}dx\right)^\sigma.\] By comparing $\gamma$ and $\gamma_{det}$, we get that $w(K)\leq w_{det}(K)$ for each $K\in \calh^{++}_{e_0}$. 
\end{remark}

\section{Optimal feedback control strategies}\label{sec:4}
The aim of this section is to provide the optimal controls (in both open-loop and closed-loop form) using a verification type argument.
The proof of this result is different when $\sigma\in(0,1)$ or $\sigma \in (1,\infty)$; hence we divide it into two different subsections. Here is the statement of this key result.

\begin{theorem}\label{teo:verification_teo}
Let Assumptions {\rm \ref{ass:state}} and {\rm \ref{ass:cost}} hold true.
For $\sigma \in(0,1)\cup (1,\infty)$ the value function of $\mathbf{({P_0})}$ is
\[V_{S,e_0}(K_0)=w(K_0)=\gamma\frac{\langle K_0,e_0 \rangle^{1-\sigma}}{1-\sigma},\]
where $\gamma$ satisfies identity \eqref{eq:gamma}.
For every initial datum $K_0\in \calh_{e_0}^{+}$
there exists a unique optimal control strategy $c^*_{K_0}$ given by (in open-loop form)
\begin{equation}\label{eq:optimal_control}
c^*(t)=\langle K_0,e_0\rangle e^{gt}e^{\alpha_0\beta_0(t)}\left(\frac{f}{\gamma N e_0}\right)^{1/\sigma},
\end{equation}
where $g:=\frac{\lambda_0-\rho}{\sigma}
-\frac12 \alpha_0^2(2-\sigma)$.
In closed-loop form we have
\begin{equation}\label{eq:optimal_control_feedback}
c^*_{K_0}(t)= G(K^{K_0}(t)),
\quad \hbox{where} \quad
G(K):=\langle K,e_0\rangle
\left(\frac{f}{\gamma N e_0}\right)^{1/\sigma}.
\end{equation}
Moreover, the corresponding optimal capital with initial condition $K_0$ is $K^{K_0,*}$, which is the unique solution to the linear stochastic PDE,
\begin{equation}\label{eq:optimal_pde}
dK(t)=\mathcal{L}K(t)-
\<K(t),e_0\>_H
\left(\frac{f}{\gamma N e_0}\right)^{1/\sigma}
N+B(K(t))dW(t),
\end{equation}
with initial condition $K(0)=K_0$.
\end{theorem}

\subsection{\texorpdfstring{\boldmath}{}The case \texorpdfstring{$\sigma\in(0,1)$}{sigmasmall}}

\begin{lemma}\label{lem:limit_equality}
Let Assumptions {\rm \ref{ass:state} and \ref{ass:cost}} hold true, and let $\sigma \in(0,1)$. Let $w$ be defined as in \eqref{eq:HJB_particular solution}, and let $\tau$ be the first exit time of $K^{K_0,c}(t)$ from $\mathcal{H}_{e_0}^{++}$, i.e.,
\begin{equation}\label{eq:tau}
\tau:=\inf\{t>0\,\,:\,\, \langle K^{K_0,c}(t),e_0 \rangle =0\}.
\end{equation}
Then the following limit holds:
\[
\lim_{t\to\infty}\E \left[ e^{-\rho (t\wedge \tau)}w(K^{K_0,c}(t\wedge \tau))\right]=0
\]
for each $K_0 \in \calh^{+}_{e_0}$, $c\in {\mathcal{A}_{S,e_0}}(K_0)$.
\end{lemma}

\begin{proof}
Let us consider $K_0 \in \partial \mathcal{H}^+_{e_0}$, i.e., $\langle K_0,e_0\rangle=0$. If $K_0$ lies on the boundary, then the only admissible control is the null one. Indeed, denoting $X_0(t):=\langle K^{K_0,c}(t),e_0\rangle$, we see that the process $X_0$ is the solution of the following SDE:
\begin{equation}\label{eq:limit_equality}
\begin{cases}
dX_0(t)=\lambda_0X_0(t)dt-\langle c(t)N,e_0\rangle dt+\alpha_0X_0(t)d\beta_0(t),\quad t>0,\\
X_0(0)=0.
\end{cases}
\end{equation}
Then, by section 4 in \cite{kloeden2013numerical}, we can derive an explicit solution for this equation,
\[X_0(t)=-\int_0^t\Psi(t)\Psi(s)^{-1}\langle c(s)N,e_0\rangle ds\]
with $\Psi(t)=\exp\big[\big(\lambda_0-\frac{\alpha_0^2}{2}\big)t+\alpha_0\beta_0(t)\big]$. Since $\Psi(t)\Psi(s)^{-1}, N,e_0>0$, the only control $c$ that makes $X_0$ nonnegative is the null control, 
\begin{equation}\label{eq:zero_boundary_boundary}
c(t)\equiv 0,\quad X_0(t)\equiv 0,
\end{equation}
 $t$-a.e., $\PP$-a.s. In conclusion, if $K_0\in\partial\mathcal{H}_{e_0}^{++}$, then the statement of the lemma is trivially verified.  
Let $K_0 \in \calh^{++}_{e_0}$, $c\in {\mathcal{A}_{S,e_0}}(K_0)$.
Notice that the only admissible control, after which $\langle K^{K_0,c}(t),e_0\rangle$ touches the boundary, is the null control. Indeed, if $\tau(\omega)<\infty$, then for $t>\tau(\omega)$,
\begin{align*}
X_0(t)&=-\tilde{\Psi}(t,\tau(\omega),\omega)\int_0^\tau\Psi(\tau(\omega),\omega)\Psi(s,\omega)^{-1}\langle c(s,\omega)N,e_0\rangle ds\\&-\int_{\tau(\omega)}^t\Psi(t,\omega)\Psi(s,\omega)^{-1}\langle c(s,\omega)N,e_0\rangle ds\\
&=-\tilde{\Psi}(t,\tau(\omega),\omega)X_0(\tau)-\int_{\tau(\omega)}^t\Psi(t,\omega)\Psi(s,\omega)^{-1}\langle c(s,\omega)N,e_0\rangle ds\\
&=-\int_{\tau(\omega)}^t\Psi(t,\omega)\Psi(s,\omega)^{-1}\langle c(s,\omega)N,e_0\rangle ds
\end{align*}
with $\tilde{\Psi}(t,\tau(\omega),\omega):=e^{[(\lambda_0-\frac{\alpha_0^2}{2})(t-\tau)+\alpha_0(\beta_0(t)-\beta(\tau))]}$. Then by repeating the above observation on the integral, we conclude that 
\begin{equation}\label{eq:zero_boundary}
c(t)\mathds{1}_{\{t>\tau\}}\equiv 0,\quad X_0(t)\mathds{1}_{\{t\geq\tau\}}\equiv 0,
\end{equation} $dt\times \PP$-a.e. in $\mathbb{R}^+\times \Omega$. 
We introduce the family of stopping times $\tau_n$, such that $\tau_n\nearrow \tau$,
\begin{equation}\label{eq:tau_n}
\tau_n=\inf\bigg\{t\geq0\,\,\text{s.t.}\,\, X_0(t)\leq \frac{1}{n}\bigg\},
\end{equation}
and the function $F_\sigma:\R^+\to\R$, $F_\sigma(x)=\gamma\frac{x^{1-\sigma}}{1-\sigma}$.
By applying the Ito formula to $e^{-\rho t}F_\sigma\left(X_0(t)\right)$ up to the stopping time $\tau_n$, we have
\begin{align*}
d\left(e^{-\rho t}F_\sigma (X_0(t))\right)
&=-\rho e^{-\rho t}F_\sigma\left(X_0(t)\right) dt+e^{-\rho t}F'(X_0(t))dX_0(t)\\
&\hspace{3cm}+\frac12\alpha_0^2X_0^2(t)
e^{-\rho t}F''_\sigma(X_0(t))dt\\
&=\gamma\lambda_0e^{-\rho t}X_0(t)^{1-\sigma}dt-\gamma\langle c(t)N,e_0\rangle e^{-\rho t}X_0(t)^{-\sigma}dt\\
&\hspace{2cm}+\gamma\alpha_0e^{-\rho t}X_0^{1-\sigma}d\beta_0(t)-\frac12\sigma\gamma\alpha_0^2e^{-\rho t}X^{1-\sigma}_0(t)dt.
\end{align*}
Since $c\in{\mathcal{A}_{S,e_0}}(K_0)$, the term involving the control is positive; hence, we write the above in its integral formulation and take the average on both sides. Since the stochastic integral, stopped at the time $\tau_n$, is a martingale, then its average is zero:
\begin{multline*}
\E\left[e^{-\rho (t\wedge \tau_n)}F_\sigma (X_0(t\wedge \tau_n))\right]
\leq\E\left[F_\sigma (X_0(0))\right]
\\+\E\left[\int_0^{t\wedge \tau_n}\left(-\rho+\left(\lambda_0-\frac12\sigma\alpha_0^2\right)(1-\sigma)\right)
e^{-\rho s}F_\sigma(X_0(s))ds\right].
\end{multline*}
Now we let $N\to \infty$. By the dominated convergence theorem, the two averages on the left-hand side  and the right-hand side converge. Then,
\begin{multline*}
\E\left[e^{-\rho (t\wedge \tau)}F_\sigma (X_0(t\wedge \tau))\right]
\leq\E\left[F_\sigma (X_0(0))\right]
\\
+\E\left[\int_0^{t\wedge \tau}\left(-\rho+\left(\lambda_0-\frac12\sigma\alpha_0^2\right)(1-\sigma)\right)
e^{-\rho s}F_\sigma(X_0(s))ds\right].
\end{multline*}
For \eqref{eq:zero_boundary}, we can rewrite the second term on the right-hand side as an integral on the entire interval $[0,t]$. Then,
\begin{multline*}
\E\left[e^{-\rho (t\wedge \tau)}F_\sigma (X_0(t\wedge \tau_n))\right]
\leq\E\left[F_\sigma (X_0(0))\right]
\\
+\int_0^{t}\left(-\rho+\left(\lambda_0-\frac12\sigma\alpha_0^2\right)(1-\sigma)\right)
\E\left[e^{-\rho (s\wedge \tau)}F_\sigma(X_0(s\wedge \tau))\right] ds.
\end{multline*}
By Gronwall's lemma (see Proposition D.29 in \cite{fabbri2017stochastic}), we get that
\begin{equation}\label{eq:limit_0}
0\leq \E\left[e^{-\rho (t\wedge \tau)}F_\sigma (X_0(t))\right]\leq F_\sigma (X_0(0))e^{-\left[\rho-\lambda_0(1-\sigma)+\frac12\alpha_0^2\sigma(1-\sigma)
\right] t}.
\end{equation}
The result then follows from Assumption \ref{ass:cost}(i).
\end{proof}

\begin{proposition}\label{prop:verification_sigma_small}
Let Assumptions {\rm \ref{ass:state}--\ref{ass:cost}} hold true, and let $\sigma \in(0,1)$.
Let $w$ be defined as in \eqref{eq:HJB_particular solution}. Let $\tau$ be the first exit time from $\mathcal{H}_{e_0}^{++}$. 
The following fundamental identity holds: for each $K_0\in\mathcal{H}^{++}_{e_0}$, $c\in {\mathcal{A}_{S,e_0}}(K_0)$,
\begin{equation}\label{eq:fundid}
    w(K_0)=J_S(c)
    +\E \left[\int_0^\tau e^{-\rho s}\left[
    H_{MAX}(\nabla w(K))-
     H_{CV}\left(\nabla w(K(s));c(s)\right)\right] ds\right].
\end{equation}
\end{proposition}

\begin{proof}
We first apply the Ito formula to $\Phi(t):=e^{-\rho t}w(K^{K_0,c}(t))$  up to the stopping time $\tau_n$, defined in \eqref{eq:tau_n}. We get
\begin{multline*}
    d \Phi(t)=-\rho e^{-\rho t}w(K^{K_0,c}(t))dt+e^{-\rho t}\langle \nabla w(K^{K_0,c}(t)),\mathcal{L}K(T)\rangle dt\\
    -e^{-\rho t}\langle \nabla w(K^{K_0,c}(t)),c(t)N\rangle dt
    +e^{-\rho t}\langle \nabla w(K^{K_0,c}(t)),B(K^{K_0,c}(t))dW_t\rangle \\
    +e^{-\rho t}\frac12\text{Tr}[D^2w(K(t))B(K^{K_0,c}(t))B(K^{K_0,c}(t))^*].
\end{multline*}
Given that $w$ is a classical solution of \eqref{eq:HJB},
\begin{multline*}
    d \Phi(t)=-e^{-\rho t}\Big(\langle \nabla w(K^{K_0,c}(t)),c(t)N\rangle dt
    +\langle \nabla w(K^{K_0,c}(t)),B(K^{K_0,c}(t))dW_t\rangle \\-\sup_{c\in H^+}\{\mathcal{U}(c(t))-\langle c(t)N,\nabla w(K^{K_0,c}(t))\rangle\}\Big).
\end{multline*}
We write this in the integral formulation adding and subtracting $\int_0^{t\wedge \tau_n} e^{-\rho s}\mathcal{U}(c(s))ds$:
\begin{align*}
    &e^{-\rho (t\wedge \tau_n)}w(K^{K_0,c}(t\wedge \tau_n))-w(K_0)\\
    &=-\int_0^{t\wedge \tau_n}e^{-\rho s}\mathcal{U}(c(s))ds
  +\int_0^{t\wedge \tau_n} e^{-\rho s}\left(\mathcal{U}(c(s))-\langle \nabla w(K(s)),c(s)N\rangle\right) ds
 \\
    & -\int_0^{t\wedge \tau_n}e^{-\rho s}\sup_{c\in H^+}\{\mathcal{U}(c(s))-\langle c(s)N,\nabla w(K^{K_0,c}(s))\rangle\}ds\\
    & +\int_0^{t\wedge \tau_n}e^{-\rho s}\langle \nabla w(K^{K_0,c}(s)),B(K^{K_0,c}(s))dW_s\rangle.
\end{align*}
We now take the average. Since the integral with respect to the Brownian motion up to the time $\tau_n$ is a martingale, and recalling the definition of
$\mathcal{U}^*$, we get that
\begin{multline*}
    \E \left[ e^{-\rho (t\wedge \tau_n)}w(K^{K_0,c}(t\wedge \tau_n))\right]-w(K_0)=
     -\E\left[\int_0^{t\wedge \tau_n}e^{-\rho s}\mathcal{U}(c(s))ds\right]
\\
    +\E \left[\int_0^{t\wedge \tau_n} e^{-\rho s}\left(\mathcal{U}(c(s))ds-\langle \nabla w(K^{K_0,c}(s)),c(s)N\rangle\right)
     -\mathcal{U}^*(\nabla w(K^{K_0,c}(s)))ds\right].
\end{multline*}
We let $n\to \infty$. By dominated convergence, the average on the left-hand side converges. The integrand on the right-hand side is negative; then by the monotone convergence theorem, we obtain that also the average on the right-hand side converges. Thus the following equality holds:
\begin{multline*}
    \E \left[ e^{-\rho (t\wedge \tau)}w(K^{K_0,c}(t\wedge \tau))\right]-w(K_0)=
     -\E\left[\int_0^{t\wedge \tau}e^{-\rho s}\mathcal{U}(c(s))ds\right]
\\
    +\E \left[\int_0^{t\wedge \tau} e^{-\rho s}\left(\mathcal{U}(c(s))ds-\langle \nabla w(K^{K_0,c}(s)),c(s)N\rangle\right)
     -\mathcal{U}^*(\nabla w(K^{K_0,c}(s)))ds\right].
\end{multline*}
From \eqref{eq:zero_boundary}, we can rewrite the first term on the right-hand side, neglecting the dependence on the stopping time $\tau$, as
\[\E\left[\int_0^{t\wedge \tau}e^{-\rho s}\mathcal{U}(c(s))ds\right]=\E\left[\int_0^{t}e^{-\rho s}\mathcal{U}(c(s))ds\right].\]
We now pass to the limit for $t\to \infty$. To the left-hand side we can apply Lemma \ref{lem:limit_equality}. Since both integrands on the right-hand side have constant sign, we can apply
monotone convergence. Hence we get
\begin{multline}\label{eq:last_step}
    w(K_0)
    =\E\left[\int_0^\infty e^{-\rho s}\mathcal{U}(c(s))ds\right]
\\
+\E \left[\int_0^\tau e^{-\rho s}\left(
    \mathcal{U}^*(\nabla w(K^{K_0,c}(s)))-\left(\mathcal{U}(c(s))-\langle \nabla w(K^{K_0,c}(s)),c(s)N\rangle\right)\right) ds\right],
    \notag
\end{multline}
which gives the claim by the definition of $J_S$
and \eqref{eq:HCV}--\eqref{eq:HMAX}.
\end{proof}

\begin{proof}[Proof of Theorem {\rm \ref{teo:verification_teo}}\textup{:} The case $\sigma \in(0,1)$]

If $K_0\in\partial\mathcal{H}^{++}_{e_0}$, then from \eqref{eq:zero_boundary_boundary} and
$J_S(c),w(K_0)\equiv 0$ the statement is trivially verified.

From the fundamental identity of Proposition \ref{prop:verification_sigma_small}, since the second term of the right-hand side is positive, we have
\[w(K_0)\geq J_S(c) \qquad \hbox{for each $c\in {\mathcal{A}}_{S,e_0}(K_0)$.}
\]
Now
notice that if we find a control strategy $c_{K_0}\in {\mathcal{A}}_{S,e_0}(K_0)$ such that, calling
$K^{K_0,c_{K_0}}$ the associated state trajectory, we have
\begin{equation}
\label{eq:cstarK0}
c_{K_0}(t)=\left(\frac{N\nabla w(K^{K_0}(t))}{f}\right)^{-\frac{1}{\sigma}}=\langle K^{K_0}(t),e_0\rangle\left(\gamma \frac{N e_0}{f}\right)^{-1/\sigma},
\end{equation}
then by the fundamental identity \eqref{eq:fundid}, it must hold that $w(K_0)=J_S(c_{K_0})$.
If we plug \eqref{eq:cstarK0} into the state equation \eqref{eq:K_abstractstoch}, we obtain the closed-loop equation
 \begin{equation}\label{eq:K_optimal_abstract}
dK(t)=\mathcal{L}K(t)-\langle K(t),e_0\rangle\left(\gamma \frac{N e_0}{f}\right)^{-1/\sigma}N+B(K(t))dW(t).
\end{equation}
From Assumption \ref{ass:cost}(ii), $(\gamma \frac{N e_0}{f})^{-1/\sigma}N\in \mathcal{H}$; then Assumptions \ref{ass:state} are verified, so Theorem \ref{teo:existence} holds and then there exists a unique mild solution for \eqref{eq:K_optimal_abstract}, which we call $K^{K_0,*}(t)$.
Hence the above argument implies that the couple
$(K^{K_0,*}, c^*_{K_0})$, where
\[c^*_{K_0}(t)=\left(\frac{N\nabla w(K^{K_0,*}(t))}{f}\right)^{-\frac{1}{\sigma}},
\]
is optimal once we prove that it is admissible.
We prove \eqref{eq:optimal_control}, which also implies admissibility.
Taking the inner product with $e_0$ in the equation for $K^{K_0,*}$, we have
\begin{multline}\label{eq:dynamic_optimalpath_e0}
\langle K^{K_0,*}(t),e_0\rangle=\langle K_0,e_0\rangle+\int_0^t\lambda_0\langle K^{K_0,*}(s),e_0\rangle ds
\\
-\int_0^t \langle K^{K_0,*}(s),e_0\rangle\langle \gamma^{-\frac{1}{\sigma}}f^{\frac{1}{\sigma}},
(Ne_0)^{-\frac{1-\sigma}{\sigma}}\rangle ds
+\int_0^t\alpha_0\langle K^{K_0,*}(s),e_0\rangle d\beta_0(s).
\end{multline}
Then, we can say that $\langle K^{K_0,*}(t),e_0\rangle$ is a geometric Brownian motion, and so it can be written as
\begin{equation}\label{eq:dynamic_optimalpath_e0_exp}
\langle K^{K_0,*}(t),e_0\rangle=\langle K(0),e_0\rangle e^{\tilde{g}t+\alpha_0\beta_0(t)-\frac{\alpha_0^2}{2}t},
\end{equation}
where $\tilde{g}=(\lambda_0-\langle \gamma^{-\frac{1}{\sigma}}f^{\frac{1}{\sigma}},(Ne_0)^{-\frac{1-\sigma}{\sigma}}\rangle )$.
Hence, if $K_0\in \calh^+_{e_0}$, then $K^*(t)\in \calh^+_{e_0}$, so the control is admissible, $c^*\in {\mathcal{A}}_{S,e_0}(K_0)$.
Notice that from \eqref{eq:gamma}, we get
\begin{equation}\label{eq:control_optimal_term}
\langle \gamma^{-\frac{1}{\sigma}}f^{\frac{1}{\sigma}},(Ne_0)^{-\frac{1-\sigma}{\sigma}}\rangle -\lambda_0=\frac{\rho-\lambda_0+\frac{1}{2}\sigma (1-\sigma)\alpha_0^2}{\sigma}.
\end{equation}
Then, the exponential (deterministic) rate of the optimal control strategy is
$g=\tilde{g}-\frac{\alpha_0^2}{2}=-\frac{\rho-\lambda_0+\frac{1}{2}\sigma (1-\sigma)\alpha_0^2+\frac{1}{2}\alpha_0^2\sigma}{\sigma}.$
In conclusion, combining \eqref{eq:dynamic_optimalpath_e0_exp} and \eqref{eq:optimal_control_feedback}, we obtain \eqref{eq:optimal_control}.

Concerning uniqueness we observe that if $\bar c_{K_0}$ is another optimal strategy at $K_0$, then the integrand in \eqref{eq:fundid} must be zero $s$-a.e., $\PP$-a.s.;
this in turn implies that $\bar c_{K_0}$ and the associated state trajectory $K^{K_0,\bar c_{K_0}}$ must satisfy
\eqref{eq:cstarK0}--\eqref{eq:K_optimal_abstract}, i.e., they are equal to $(K^{K_0,*}, c^*_{K_0})$ by the uniqueness of solution of \eqref{eq:K_optimal_abstract}.
\end{proof}

\subsection{\texorpdfstring{\boldmath}{}The case \texorpdfstring{$\sigma\in(1,\infty)$}{sigmabig}}

We start by recalling a fundamental result related to our  problem ($P_0$); see Theorem 2.31 in \cite{fabbri2017stochastic} (or Theorem 3.70 in \cite{fabbri2017stochastic} for an improved version).

\begin{proposition}[dynamic programming principle]\label{prop:DPP}%
Let Assumptions {\rm \ref{ass:state}--\ref{ass:cost}} hold.
The value function $V$ satisfies the following: for each $0<t<\infty$,
\begin{equation}\label{eq:DPP}
V_{S,e_0}(K_0)=\sup_{c\in {\mathcal{A}}_{S,e_0}(K_0)}\left[\E\left[ \int_0^te^{-\rho s}\mathcal{U}(c(s))ds\right]+e^{-\rho t}V_{S,e_0}(K^{K_0,c}(t))\right].
\end{equation}
\end{proposition}

\begin{proof}
The proof of this result is an adaptation of the proof of Theorem 3.70 in \cite{fabbri2017stochastic}. The only differences are that in the aforementioned theorem, 
\begin{itemize}
\item the horizon is finite;
\item the current cost is state-dependent and uniformly bounded in the control.
\end{itemize}
The first difference can be overcome by standard shift arguments as is done, e.g., in section 2.4 of \cite{fabbri2017stochastic}. With regard to the second difference, in the proof of Theorem 3.70 in
\cite{fabbri2017stochastic}, such boundedness is used to apply dominated convergence inside the integral. In our case, due to the specific form of the functional, we can apply monotone convergence.
\end{proof}

In this case, we need to change strategy to prove the verification theorem. For such values of $\sigma$, the value function $V$ is negative, and the same holds for the solution $w$ of the HJB equation found in Theorem \ref{th:HJB_particular solution}. Moreover, $V$ may be equal to $-\infty$.
This fact does not allow us to repeat the proof of Lemma \ref{lem:limit_equality} in this case. In particular, since $\sigma \in (1,\infty)$, the first inequality in \eqref{eq:limit_0} does not hold.
Our approach here is related to that  of \cite[section 4.6]{BGP}.
First we  prove a lower bound for $V$, i.e., that $V\ge w$, which, in particular, says that $V$ is never equal to $-\infty$ on $\calh^{++}_{e_0}$.
Then, using homogeneity and Proposition \ref{prop:DPP}, we prove that equality indeed holds.



\begin{lemma}\label{lem:V_not_infinite}
Let Assumptions {\rm \ref{ass:state}--\ref{ass:cost}} hold. If $\sigma \in(1,\infty)$, then for each $K_0\in \mathcal{H}^{++}_{e_0}$ there exists a couple $(c^*_{K_0},K^{K_0,c^*_{K_0}})$, given by an admissible control $c^*_{K_0}$, and by state variable $K^{K_0,c^*_{K_0}}$, such that
\[c^*_{K_0}(t)=\left(\frac{N\nabla w(K^{K_0,c^*_{K_0}}(t))}{f}\right)^{-\frac{1}{\sigma}}\]
and such that the functional $J_S(c^*_{K_0})$ is finite, namely
\[0\geq V_{S,e_0}(K_0)\geq J_S(c^*)>-\infty,\]
and in particular $J_S(c^*_{K_0})=w(K_0)$, where $w$ is defined as in \eqref{eq:HJB_particular solution}.
\end{lemma}

\begin{proof}
We consider the control, expressed in a feedback form,
\[c^*_{K_0}(t)=\left(\frac{N\nabla w(K^{K_0,c^*_{K_0}}(t))}{f}\right)^{-\frac{1}{\sigma}}=\langle K^{K_0,c^*_{K_0}}(t),e_0\rangle\left(\gamma \frac{N e_0}{f}\right)^{-1/\sigma}.\]
By plugging this control into the state equation \eqref{eq:K_abstractstoch}, we obtain the closed-loop equation \eqref{eq:K_optimal_abstract}. From Theorem \ref{teo:existence} there exists a unique mild solution for the closed-loop equation $K^{K_0,c^*_{K_0}}(t)$.
In order to see that this control is admissible, we just need to repeat the computation proposed in \eqref{eq:dynamic_optimalpath_e0_exp}. The control $c^*_{K_0}$ can be rewritten in an open-loop form; see \eqref{eq:optimal_control}. Then by plugging \eqref{eq:optimal_control_feedback} into the functional $J_S$, from \eqref{eq:optimal_control} we get
\begin{align}\label{eq:J_c_optimal}
J_S(c^*_{K_0})&=\langle \gamma^{-\frac{1}{\sigma}}f^{\frac{1}{\sigma}},(Ne_0)^{-\frac{1-\sigma}{\sigma}}\rangle \gamma\frac{\langle K^{K_0,*}(0),e_0\rangle^{1-\sigma}}{1-\sigma}\int_0^\infty e^{-\rho t}\E\left[e^{gt+\alpha_0\beta_0(t)}\right] dt\nonumber\\
&=\langle \gamma^{-\frac{1}{\sigma}}f^{\frac{1}{\sigma}},(Ne_0)^{-\frac{1-\sigma}{\sigma}}\rangle  \gamma\frac{\langle K^{K_0,*}(0),e_0\rangle^{1-\sigma}}{1-\sigma}\int_0^\infty e^{-\frac{\left(\rho-\lambda_0(1-\sigma)+\frac12\alpha_0^2\sigma(1-\sigma)\right)}{\sigma} t}dt\nonumber\\
&= \gamma\frac{\langle K^{K_0,*}(0),e_0\rangle^{1-\sigma}}{1-\sigma}.\nonumber
\end{align}
 \end{proof}

\begin{proposition}\label{prop:explicit_value function}
Let Assumptions {\rm \ref{ass:state}--\ref{ass:cost}} hold. If $\sigma \in(1,\infty)$, then the value function $V$ satisfies the following identity for $K_0\in \mathcal{H}^{++}_{e_0}$\textup{:}
\[V_{S,e_0}(K_0)=\eta\frac{\langle K_0,e_0\rangle^{1-\sigma}}{1-\sigma},\]
where $\eta\geq 0$.
\end{proposition}

{\em Proof.} {\it Step} 1.\ First, we prove that the value function $V_{S,e_0}(K)$ is a function of $\langle K,e_0\rangle$. If we pick two elements in $\mathcal{H}$, $K_1,K_2\in \mathcal{H}$ such that
\begin{equation}\label{eq:Ke0_equality}
\langle K_1,e_0\rangle=\langle K_2,e_0\rangle,
\end{equation}
 then
\begin{equation}\label{eq:value_function_equality}
V_{S,e_0}(K_1)=V_{S,e_0}(K_2).
\end{equation}
If we prove that $\mathcal{A}_{S,e_0}(K_1)=\mathcal{A}_{S,e_0}(K_2)$, then \eqref{eq:value_function_equality} is proved. By writing the weak formulation for $K(t)$, with test function $e_0$, we have that $X^{K_0,c}(t)=\langle K^{K_0,c}(t),e_0\rangle$ is the unique solution of the following equation:
\begin{equation}\label{eq:X}
\begin{cases}
dX(t)=\lambda_0X(t) dt-\langle c(t)N,\phi\rangle dt+\alpha_0X(t)d\beta_0(t),\\
X(0)=\langle K_0,e_0\rangle.
\end{cases}
\end{equation}
 Assumption \eqref{eq:Ke0_equality} means that the processes $X^{K_1,c}(t)$, $X^{K_2,c}(t)$ satisfy the same equation and the same initial condition. Then by uniqueness of solution of \eqref{eq:X}, we have that $X^{K_1,c}(t)=X^{K_2,c}(t)$ for each $t\in[0,T]$, a.s., and we can conclude that the two sets  $\mathcal{A}_{S,e_0}(K_1), \mathcal{A}_{S,e_0}(K_2)$ coincide. Since from Lemma \ref{lem:V_not_infinite}, $0\geq V>-\infty$, there exists a function $F:\R\to \R$ such that
\[V_{S,e_0}(K)=F(\langle K,e_0\rangle).\]

{\it Step}~2.
The next lines are devoted to the proof of the $(1-\sigma)$-homogeneity of the function $F$, i.e., given $a>0$,
\[V_{S,e_0}(aK)=a^{1-\sigma} V_{S,e_0}(K).\]
Let us assume that $\mathcal{A}_{S,e_0}(aK)=a\mathcal{A}_{S,e_0}(K)$ holds; then
\begin{multline*}
V_{S,e_0}(aK)=\sup_{c\in \mathcal{A}_{S,e_0}(aK)} J_S(c)=\sup_{c\in a\mathcal{A}_{S,e_0}(K)} J_S(c)=\sup_{a^{-1}c\in \mathcal{A}_{S,e_0}(K)} J_S(c)\\=\sup_{\tilde{c}\in \mathcal{A}_{S,e_0}(K)} J_S(a\tilde{c})
=a^{1-\sigma}\sup_{\tilde{c}\in \mathcal{A}_{S,e_0}(K)} J_S(\tilde{c})
=a^{1-\sigma}V_{S,e_0}(K).
\end{multline*}
In order to conclude, we need to prove the equality between the sets of controls $a\mathcal{A}_{S,e_0}(K_0)$, $\mathcal{A}_{S,e_0}(aK_0)$. First, we prove that if $c\in \mathcal{A}_{S,e_0}(aK_0)$, then $c=a\cdot \tilde{c}$ with $\tilde{c}\in \mathcal{A}_{S,e_0}(K_0)$. By definition, $c\in \mathcal{A}_{S,e_0}(aK_0)$ means that $\langle K^{aK_0,c}(t), e_0\rangle>0$. By linearity of \eqref{eq:K_abstract}, we get that $a^{-1}K^{aK_0,c}(t)=K^{K_0,a^{-1}c}(t)$. In this way we have proved that $a^{-1}c=\tilde{c}$ with $\tilde{c}\in\mathcal{A}_{S,e_0}(K_0)$. To prove the reverse inclusion $a\mathcal{A}_{S,e_0}(K_0)\subseteq\mathcal{A}_{S,e_0}(aK_0)$, a similar argument is used. We want to prove that given  $\tilde{c}\in\mathcal{A}_{S,e_0}(K_0)$, then $a\cdot \tilde{c}\in \mathcal{A}_{S,e_0}(aK_0)$. By definition, if $\tilde{c}\in\mathcal{A}_{S,e_0}(K_0)$, then $\langle K^{K_0,\tilde{c}}(t), e_0\rangle>0$ by linearity $a\cdot K^{K_0,\tilde{c}}(t)=K^{aK_0,a\tilde{c}}(t)$, which means that $a\cdot \tilde{c}\in \mathcal{A}_{S,e_0}(aK_0)$.

Since the function $F$ is negative and $(1-\sigma)$-homogeneous, we can write the value function as
\begin{align*}
V_{S,e_0}(K)=\eta\frac{\langle K,e_0\rangle^{1-\sigma}}{1-\sigma},
\end{align*}
with $\eta\geq 0$.

\begin{lemma}\label{lem:value function_no0}
Let Assumptions {\rm \ref{ass:state}--\ref{ass:cost}} hold. If $\sigma \in(1,\infty)$, the value function on $\mathcal{H}^{++}_{e_0}$ is not always $0$, i.e., $V_{S,e_0}\neq 0$.
\end{lemma}

\begin{proof}
We assume by contradiction that $V_{S,e_0}\equiv 0$. By the dynamic programming principle, in Proposition \ref{prop:DPP}, we have that for each $T>0$
\begin{align*}
V_{S,e_0}(K_0)=\sup_{c\in \mathcal{A}_{S,e_0}(K_0)}\E\left[ \int_0^Te^{-\rho t}\mathcal{U}(c(s))ds\right]+V_{S,e_0}(K^{K_0,c}(T))\\
\Rightarrow \sup_{c\in \mathcal{A}_{S,e_0}(K_0)}\E\left[ \int_0^Te^{-\rho t}\mathcal{U}(c(s))ds\right]=0.
\end{align*}
Given $n\in\N$, we choose a particular sequence of control $c_n$ such that
\begin{equation}\label{eq:assumption_cn}
-\frac{1}{n}\leq\E\left[ \int_0^Te^{-\rho t}\mathcal{U}(c_n(s))ds\right]\leq 0,\\ \end{equation}
and a particular sequence of initial condition for the state variable $K_n(0)\in L^2(S^2)$ such that
$\langle K_n(0),e_0\rangle\leq \frac{1}{n}.$
Now we write an equation satisfied by the capital variable with initial conditions $K_n(0)$ and control $c_n$, in its weak formulation with test function $e_0$:
\begin{multline*}
\langle K(t),e_0\rangle=\langle K_n(0),e_0\rangle+\lambda_0\int_0^{t}\langle K(s),e_0\rangle ds-\int_0^{t}\langle c_n(s)N,e_0\rangle ds\\
+\int_0^{t}\alpha_0\langle K(s),e_0\rangle d\beta_0(s)\quad \forall t\in [0,T].
\end{multline*}
We pass to the average,
\[
\E\left[\langle K(t),e_0\rangle\right]=\langle K_n(0),e_0\rangle+\lambda_0\int_0^{t}\E\left[\langle K(s),e_0\rangle \right] ds-\int_0^{t}\E\left[\langle c_n(s)N,e_0\rangle \right]ds.
\]
In terms of the derivative,
\[\frac{d}{dt}\E\left[\langle K(t),e_0\rangle\right]=\lambda_0\E\left[\langle K(t),e_0\rangle\right]-\E\left[\langle c_n(t)N,e_0\rangle \right];\]
then
\[0\leq\E\left[\langle K(t),e_0\rangle\right]=e^{\lambda_0 t}\langle K_n(0),e_0\rangle-\int_0^{t}e^{\lambda_0(t-s)}\E\left[\langle c_n(s)N,e_0\rangle \right]ds.\]
The previous inequality implies
\[\int_0^{t}e^{\lambda_0(t-s)}\E\left[\langle c_n(s)N,e_0\rangle \right]ds\leq e^{\lambda_0 t}\langle K_n(0),e_0\rangle.\]
After some manipulations on the exponential function, we can rewrite the previous inequality as
\begin{equation}\label{eq:inequality_1}
0\leq\int_0^{t}e^{\lambda_0s}\E\left[\langle c_n(s)N,e_0\rangle \right]ds\leq \langle K_n(0),e_0\rangle\leq \frac{1}{n}\quad \forall t\in [0,T].
\end{equation}
Then, we define $\mathcal{C}_n(s)=\langle c_n(t)N,e_0\rangle$, and as a straightforward consequence of \eqref{eq:inequality_1}, we get that this quantity converges to zero, i.e.,
\[\mathcal{C}_n(s)\to 0\,\,\, \text{in}\,\,L^1(\Omega\times [0,T]).\]
Then there exists a subsequence $n_k$ such that
\[\mathcal{C}_{n_k}(s)\to 0\,\,\, \text{a.e. in}\,\,\Omega\times[0,T].\]
Since $N,e_0>0$,  we can extract a subsequence of $n_k$, $n_{k_h}$ such that
\[c_{n_{k_h}}(t,x,\omega)\to0\,\,\, \text{a.e. in}\,\,\Omega\times[0,T]\times S^1.\]
Since $\sigma >1$, $\frac{c_{n_{k_h}}^{1-\sigma}}{1-\sigma}\to -\infty$ for $h\to \infty$, which contradicts \eqref{eq:assumption_cn}.
\end{proof}

\begin{theorem}
Let Assumptions {\rm \ref{ass:state}--\ref{ass:cost}} hold. If $\sigma \in(1,\infty)$, then the value function $V$ is a classical solution of the HJB equation \eqref{eq:HJB}, and it satisfies identity
\[V_{S,e_0}(K_0)=\gamma\frac{\langle K_0,e_0\rangle^{1-\sigma}}{1-\sigma}\]
for each $K_0\in \mathcal{H}^{++}_{e_0}$ where $\gamma$ satisfies \eqref{eq:gamma}.
\end{theorem}

\begin{proof}
In Proposition \ref{prop:explicit_value function}, we have derived an explicit expression for the value function $V$. Since $V$ satisfies all the regularity assumptions required to be a classical solution, we plug into  \eqref{eq:HJB} the value function $V_{S,e_0}(K)$, and we find that the only two admissible values are $\eta=0$ and $\eta=\gamma$. From Lemma \ref{lem:value function_no0}, we can exclude the null case.
\end{proof}

\begin{proof}[Proof of Theorem {\rm \ref{teo:verification_teo}}\textup{:} The case $\sigma \in(1,\infty)$]
We start with the study of the boundary. We pick $K_0\in \partial \mathcal{H}^+_{e_0}$. By repeating the argument proposed in Lemma \ref{lem:limit_equality} to prove \eqref{eq:zero_boundary_boundary}, we get that the only admissible control is the null one. Then $V_{S,e_0}(K_0),w(K_0)=-\infty$ and the identity is trivially verified. 
Now we investigate the case $K_0\in \mathcal{H}_{e_0}^{++}$. 
 In Proposition \ref{prop:explicit_value function}, we have derived an explicit expression for the value function $V_{S,e_0}$. Since $V_{S,e_0}$ satisfy all the regularity assumptions required to be a classical solution, we plug into  \eqref{eq:HJB} the value function $V_{S,e_0}(K)$, and we find that the only two admissible values are $\eta=0$ and $\eta=\gamma$. From Lemma \ref{lem:value function_no0}, we can exclude the null case and $V_{S,e_0}(K)=w(K)$.
From Lemma \ref{lem:V_not_infinite}, $J_S(c^*)=w(K)$, where the optimal control $c^*$ is the same as the one treated in the first case.

Since the optimal control is the same as in the previous case, $\sigma\in(0,1)$, and since all the arguments on the optimal path presented in Theorem \ref{teo:verification_teo} do not depend on the value of $\sigma$, the proof is concluded.
\end{proof}

\section{Long time behavior of the optimal path}\label{sec:5}
The aim of this section is to explore the asymptotic distribution of the optimal path.

In subsection \ref{subsec:ltb_op_detrended}, we will prove that for the detrended optimal path  $K^{K_0,*}_g(t):=e^{-gt-\alpha_0\beta_0(t)}K^{K_0,*}(t)$ with $K_0\in \mathcal{H}_{e_0}^{++}$  there are infinitely many invariant measures in $\mathcal{H}$; in particular, given an initial condition $K_0\in \mathcal{H}_{e_0}^{++}$, then the solution $K^{K_0,*}_g(t)$ converges weakly to an invariant measure which depends on $K_0$. Similar results on the invariant measure can be found in \cite{farkas2020class} and \cite{vargiolu1999invariant}. 
In subsection \ref{subsec:ltb_op} the result for the nondetrended optimal path is presented. In this case we will prove that there exists a unique invariant measure, a Dirac mass centered in the null process, $\delta_{0}$.
\begin{remark}\label{rem:sup_exp_Brownian}
It is worth recalling the following result. Consider $B_1(t),B_2(t)$ as two independent Brownian motions and  $t\geq 0$, $\mu,\sigma_1,\sigma_2\geq0$ with $\sigma_1+\sigma_2>0$. We define
\[S_t=\sup_{0\leq s\leq t} \exp\left[-\mu t +\sigma_1 B_1(t)+\sigma_2B_2(t)\right].
\]
By section 3.5 in \cite{karatzas2014brownian}, for $x\in \R^+$,
\begin{equation}\label{eq:S_infty_bound1}
\PP\left(S_\infty>x\right)\leq e^{-\lambda \log(x)}
\end{equation}
with $\lambda=\frac{2\mu}{\sigma_1^2+\sigma_2^2}$.
Moreover, given $n\in \N$, the sequence $\{S_\infty<n\}$ is a sequence of increasing events; then
\[ \PP\left(S_\infty<\infty\right)=\PP\left(\bigcup_{n\in\N}\left(S_\infty<n\right)\right)=\lim_{ n\to\infty}\PP(S_\infty<n).\]
Hence, from \eqref{eq:S_infty_bound1}, we get that
$
\PP\left(S_\infty<\infty\right)=1.
$ 
\end{remark}

\subsection{Convergence in law of the detrended optimal path}\label{subsec:ltb_op_detrended}
In this section the detrended optimal path $K^{K_0,*}_g(t)$ is defined. From \eqref{eq:optimal_control}, we can see that the optimal control evolves like an  $e^{gt+\alpha_0\beta_0(t)}$. Thus, we detrend $K^{K_0,*}(t)$ with respect to this process,
\[K^{K_0,*}_g(t):=e^{-gt-\alpha_0\beta_0(t)}K^{K_0,*}(t).\]
The process $K^{K_0,*}_g(t)$ is solution of the equation
\begin{multline}\label{eq:K_g}
dK(t)=\mathcal{L}K(t)dt-\left(\int_{\mathbb
{S}^1} K(t,x)e_0(x)dx\right) \left(\gamma \frac{N e_0}{f}\right)^{-1/\sigma}Ndt-gK(t)dt\\
+\frac{\alpha_0^2}{2}K(t)dt-\alpha_0^2 \langle K,e_0\rangle e_0 dt+B(K(t))dW(t)
-\alpha_0K(t) d\beta_0(t).
\end{multline}
In order to study this process, we write $K^{K_0,*}_g(t)$ in terms of a Fourier expansion and then we study the behavior of its coefficients. 
Given the sequence of eigenfunctions $(e_n)_n$, which is an orthonormal basis of $\mathcal{H}$, we write the Fourier expansion for $K^{K_0,*}_g(t)$,
\[K^{K_0,*}_g(t)=\sum_n K^{K_0,*}_{n,g}(t)e_n, \]
where $ K^{K_0,*}_{n,g}(t)=\langle K^{K_0,*}_g(t),e_n\rangle$.

\begin{lemma}\label{lem:ltb_components}
Let Assumptions {\rm \ref{ass:state}--\ref{ass:cost}} hold. Moreover, we assume that $\lambda_1<g$, $\alpha_n(K)=\alpha_0$ for each $n\geq 0$.
Then for all $n\geq 0$, $K_{n,g}^{K_0,*}(t)$ converges in law to $K_{n,g,\infty}^{K_0,*}$,
where
\begin{equation}\label{coefficient_K_infty}
K_{n,g,\infty}^{K_0,*}=
\begin{cases}
K_0(0),\hspace{6.2cm} n=0,\\
c_n K_0(0)\int_0^\infty e^{(-g+\lambda_n-\frac{\alpha_0^2}{2})r-\alpha_0\beta_0(r)+\alpha_0\beta_n(r)}ds,\quad n\geq 1.
\end{cases}
\end{equation}
Moreover, there exists a family of processes $X_{n,g}^{K_0,*}(t)$, equal in law to $K_{n,g}^{K_0,*}(t)$, defined by
\begin{equation}\label{eq:X_n}
X_{n,g}(t)=
\begin{cases}
K_0(0),\hspace{7.5cm}n=0,\\
K_{n}(0)\mathcal{E}_{\lambda_n,g,\alpha_0}(t)+c_n K_0(0)\int_0^\infty \mathcal{E}_{\lambda_n,g,\alpha_0}(r)\mathds{1}_{[0,t]}(r)dr,\,\,\,\ \  n\geq1,
\end{cases}
\end{equation}
with

\[\mathcal{E}_{\lambda_n,g,\alpha_0}(r)=e^{\big[\big(\lambda_n-g-\frac{\alpha_0^2}{2}\big)r-\alpha_0\beta_0(r)+\alpha_0\beta_n(r)\big]},\]
that converge to $K_{n,\infty}^{K_0,*}(t)$ for each $n\in\N$ a.s., i.e.,
\[\PP\left(\lim_{t\to\infty}X_{n,g}^{K_0,*}(t)=K_{n,g,\infty}^{K_0,*}\quad \forall n\in\N\right)=1.\]
\end{lemma}

\begin{proof}
By plugging $\phi=e_n$ into the weak formulation for $K^{K_0,*}_g(t)$, we get
\begin{multline*}\label{eq:weak _formulation_K}
\langle K^{K_0,*}_g(t),e_n\rangle=\langle K(0),e_n\rangle+\int_0^t(\lambda_n-g)\langle K^{K_0,*}_g(s),e_n\rangle ds\\
-\int_0^t \langle K^{K_0,*}_g(s),e_0\rangle\Big\langle\Big(\gamma \frac{N e_0}{f}\Big)^{-1/\sigma}N,e_n\Big\rangle ds+
\int_0^t\alpha_0\langle K^{K_0,*}_g(s),e_n\rangle d\beta_n(s)\\
-\int_0^t\alpha_0\langle K^{K_0,*}_g(s),e_n\rangle d\beta_0(s)
+\frac{\alpha_0^2}{2}\int_0^t\langle K^{K_0,*}_g(s),e_n\rangle ds\\-\alpha_0^2\int_0^t\langle K^{K_0,*}_g(s),e_0\rangle\langle e_0,e_n\rangle ds.
\end{multline*}
For simplicity we omit $(K_0,*)$ in $K^{K_0,*}_{n,g}$, and we write $K_{n,g}$ in the place of $K_n^{K_0,*}$. For $n>0$, $K_{n,g}$ is solution of the equation
\begin{equation}\label{eq:K_n_g}
\begin{cases}
dK_{n,g}(t)=(\lambda_n-g+\frac{\alpha_0^2}{2})K_{n,g}(t)dt-c_nK_{0,g}(t)dt +\alpha_0K_{n,g}(t)d\beta_n(t)\\
\hspace{8cm}-\alpha_0K_{n,g}(t)d\beta_0(t),\\
K_n(0)=\langle K(0),e_n\rangle.
\end{cases}
\end{equation}
For $n=0$,
\[
dK_{0,g}(t)=\bigg(\lambda_0-c_0 -g-\frac{\alpha_0^2}{2}\bigg)K_{0,g}(t)dt.
\]
From equality \eqref{eq:gamma} and from the definition of $g$,
$\lambda_0-c_0=g+\frac{\alpha_0^2}{2}$. Thus we can rewrite the equation for $K_0$ as
\begin{equation*}
dK_{0,g}(t)=0,\quad K_0(0)=\langle K(0),e_0\rangle.
\end{equation*}
Then
$
K_{0,g}(t)=K_0(0).
$
We plug $K_{0,g}(t)$ into \eqref{eq:K_n_g}, and the equation for $K_{n,g}$ becomes
\begin{equation*}
\begin{cases}
dK_{n,g}(t)=(\lambda_n-g+\frac{\alpha_0^2}{2})K_{n,g}(t)dt-c_nK_{0}(0)dt +\alpha_0K_{n,g}(t)d\beta_n(t)\\
\hspace{8cm}-\alpha_0K_{n,g}(t)d\beta_0(t),\\
K_n(0)=\langle K(0),e_n\rangle.
\end{cases}
\end{equation*}
Since $K_{n,g}(t)$ is solution of a linear SDE with inhomogeneous constant coefficients, we can explicitly solve such an equation (see Chapter $4$ in \cite{kloeden2013numerical}): 
\begin{equation}\label{eq:K_n}
K_{n,g}(t)=\mathcal{E}_{\lambda_n,g,\alpha_0}(t)\left(K_{n}(0)+c_n K_0(0)\mathcal{E}_{\lambda_n,g,\alpha_0}(s)^{-1}ds \right),\end{equation}
with
\begin{equation*}\mathcal{E}_{\lambda_n,g,\alpha_0}=e^{\big(\lambda_n-g-\frac{\alpha_0^2}{2}\big)t-\alpha_0\beta_0(t)+\alpha_0\beta_n(t)}. 
\end{equation*}
By a change of variable, $r=t-s$, we rewrite the second term in \eqref{eq:K_n}, the one involving the time integral:
\begin{align*}
c_n K_0(0)\int_0^t \mathcal{E}_{\lambda_n,g,\alpha_0}(t)(\mathcal{E}_{\lambda_n,g,\alpha_0}(s))^{-1}(\omega)ds
\\
=K_0(0)c_n\int_0^t e^{\big[\big(\lambda_n-g-\frac{\alpha_0^2}{2}\big)(t-s)-\alpha_0(\beta_0(t)-\beta_0(s))+\alpha_0(\beta_n(t)-\beta_n(s))\big]}(\omega)ds\\
=-K_0(0)c_n\int_t^0 e^{\big[\big(\lambda_n-g-\frac{\alpha_0^2}{2}\big)r-\alpha_0(\beta_0(
t)-\beta_0(t-r))+\alpha_0(\beta_n(t)-\beta_n(t-r))\big]} (\omega)dr\\
=K_0(0)c_n\int_0^t e^{\big[\big(\lambda_n-g-\frac{\alpha_0^2}{2}\big)r-\alpha_0(\beta_0(
t)-\beta_0(t-r))+\alpha_0(\beta_n(t)-\beta_n(t-r))\big]} (\omega)dr.
\end{align*}
We introduce an auxiliary process $X_{n,g}$, such that $X_{0,g}(t)=K_0(0)$ and
\[
X_{n,g}(t)=\mathcal{E}_{\lambda_n,g,\alpha_0}(t)K_{n}(0)+c_n K_0(0)\int_0^\infty \mathcal{E}_{\lambda_n,g,\alpha_0}(r)\mathds{1}_{[0,t]}(s)ds,\quad n\geq 1.\]
First, we notice that since for each $r\geq0$, $\left(\beta(t)-\beta(t-r)\right)_{t\geq0}$ is a Brownian motion time reversed, we have that $K_{n,g}$, $X_{n,g}$ are equal in law:
\begin{equation}\label{eq:law_equal}
\mathfrak{L}\left( K_{n,g} (t)\right)=\mathfrak{L}\left( X_{n,g}(t) \right)\quad \forall t\geq 0.
\end{equation}
The aim of what follows is to prove the convergence almost surely of $X_{n,g}(t)$. Since  $\lambda_1<g$, then $\lambda_n\leq g<0$ for each $n\geq 1$. By the law of the iterated logarithm, the $t$ term in the exponential will dominate, and since $\lambda_n-g-\frac{\alpha_0^2}{2}<0$, we get that $K_n(0)\mathcal{E}_{\lambda_n,g,\alpha_0}(t)\to 0$ a.s.
A bit more demanding is the study of the other term. Our objective is to prove that there exists $\tilde{\Omega}$, a set of probability one, such that for each $\omega \in \tilde{\Omega}$, 
\begin{equation}\label{eq:limit_exp_int}
\lim_{t\to\infty}K_0(0)c_n\int_0^\infty\mathcal{E}_{\lambda_n,g,\alpha_0}(r,\omega)\mathds{1}_{[0,t]}(r)dr
=K_0(0)c_n\int_0^\infty \mathcal{E}_{\lambda_n,g,\alpha_0}(r,\omega)dr.
\end{equation}
First we just recall the obvious result on the indicator function, namely,
\[\lim_{t\to\infty}\mathds{1}_{[0,t]}(r)=1.\]
Then we verify that the sequence $\mathcal{E}_{\lambda_n,g,\alpha_0}(r,\omega)$ has uniformly absolutely continuous integrals; namely, for each $\epsilon >0$ there exists $\delta_\epsilon>0$ such that
\begin{equation}\label{eq:vitali_ii}
\int_A \mathcal{E}_{\lambda_n,g,\alpha_0}(r)\mathds{1}_{[0,t]}(r,\omega))dr<\epsilon
\end{equation}
for each Lebesgue measurable set $A$ whose measure is $\mathcal{L}(A)<\delta_\epsilon$.
By Remark \ref{rem:sup_exp_Brownian} with $\mu=\big(-\lambda_n+g+\frac{\alpha_0^2}{2}\big)$, $\sigma_1=\sigma_2=\alpha_0$, we get that there exists a constant $C_1(\omega)$ such that $\sup_{r<\infty}\mathcal{E}_{\lambda_n,g,\alpha_0}(r,\omega)<C_1(\omega)$. Thus, for all sets $A$ whose measure is $\mathcal{L}(A)<\delta_\epsilon$ with $\delta_\epsilon=\epsilon/C_1(\omega)$,
\[
\int_A \mathcal{E}_{\lambda_n,g,\alpha_0}(r,\omega)\mathds{1}_{[0,t]}(r)dr\leq C_1(\omega)\mathcal{L}(A\cap[0,t])\leq \epsilon.
\]
In conclusion we furnish an integral estimate on long times. For each $\epsilon>0$ there exists an interval $[0,\eta]$ such that
\[\int_{[0,\eta]^c} \mathcal{E}_{\lambda_n,g,\alpha_0}(r,\omega)\mathds{1}_{[0,t]}(r)dr<\epsilon.\]
Fixing a constant $a>0$ such that
$
\lambda_n-g-\frac{\alpha_0^2}{2}+a<0,
$ we split $\mathcal{E}_{\lambda_n,g,\alpha_0}(r,\omega)$:
\begin{align*}
\mathcal{E}_{\lambda_n,g,\alpha_0}(r,\omega)&=e^{\big(\lambda_n-g-\frac{\alpha_0^2}{2}+a\big)r}e^{-ar-\alpha_0\beta_0(r)+\alpha_0\beta_n(r)}(\omega)\\
&=e^{\big(\lambda_n-g-\frac{\alpha_0^2}{2}+a\big)r}\mathcal{E}_{a,\alpha_0,s}(r,\omega).
\end{align*}
By Remark \ref{rem:sup_exp_Brownian} with $\mu=a$, $\sigma_1=\sigma_2=\alpha_0$, we get that there exists a constant $C_2(\omega)$ such that $\sup_{r<\infty}\mathcal{E}_{a,\alpha_0,s}(r,\omega)<C_2(\omega)$. Then, by choosing $\eta>0$ such that 
$$\frac{e^{\big(\lambda_n-g-\frac{\alpha_0^2}{2}+a\big)\eta}}{\big(-\lambda_n+g+\frac{\alpha_0^2}{2}-a\big)} <\epsilon/C_2(\omega),$$
we conclude that
\begin{align*}
\int_{[0,\eta]^c} \mathcal{E}_{\lambda_n,g,\alpha_0}(r,\omega)\mathds{1}_{[0,t]}(r)dr&\leq C_2(\omega)\int_{[0,\eta]^c} e^{\big(\lambda_n-g-\frac{\alpha_0^2}{2}+a\big)r} dr\\
&=C_2(\omega)\frac{e^{\big(\lambda_n-g-\frac{\alpha_0^2}{2}+a\big)\eta}}{\big(-\lambda_n+g+\frac{\alpha_0^2}{2}-a\big)}<\epsilon.
\end{align*}
Finally, by Vitali's convergence theorem, we have proved that 
\[
\PP\left(\lim_{t\to\infty} X_{n,g}^{K_0,*}(t)=K_{n,g,\infty}^{K_0,*}\right)=1\quad \forall n\geq0,
\]
where
\begin{equation*}
K_{n,g,\infty}^{K_0,*}=
\begin{cases}
K_0(0),\hspace{9cm} n=0,\\
K_0(0)c_n\int_0^\infty e^{(-g+\lambda_n-\frac{\alpha_0^2}{2})r-\alpha_0(\beta_0(t)-\beta_0(t-r))+\alpha_0(\beta_n(t)-\beta_n(t-r))}ds,\quad n\geq 1.
\end{cases}
\end{equation*}
Since convergence almost surely implies convergence in law and since $X_{n,g}$ and $K_{n,g}$ are equal in law \eqref{eq:law_equal}, we can derive the convergence in law of $K_{n,g}$, i.e.,
\[\int \phi(x)\mathfrak{L}(K_{n,g}(t))(dx)=\int \phi(x)\mathfrak{L}(X_{n,g}(t))(dx)\to \int \phi(x)\mathfrak{L}(X_{n,g,\infty})(dx),\quad t\to\infty,\]
for all $\phi\in C_b(\R)$.
\end{proof}

\begin{theorem}
\label{teo:convtoK_infty}
Let Assumptions {\rm \ref{ass:state}--\ref{ass:cost}} hold. Moreover, we assume that $\lambda_1<g$, $\alpha_n(K)=\alpha_0$ for each $n\geq 0$. Then,  the detrended optimal path $K^{K_0,*}_g(t)$ converges in law to  $K_{g,\infty}^{K_0,*}(t):=\sum_nK_{n,g,\infty}^{K_0,*}e_n$, i.e.,
\[\mathfrak{L}(K_{g}^{K_0,*}(t))\to \mathfrak{L}(K_{g,\infty}^{K_0,*}).\]
\end{theorem}

\begin{proof}
We introduce the  operators $P_N$, $T_N$, and $\Pi_N$: 
\[P_N:\, \mathcal{H}\longrightarrow \mathcal{H},\,\, P_Nf= \sum_{n=0}^{N-1} \langle f,e_n\rangle e_n,\]
\[ \Pi_N:\, \mathcal{H}\longrightarrow \R^N,\,\, \Pi_Nf= \left( \langle f,e_1\rangle,\dots,\langle f,e_N\rangle\right),\]
\[T_N:\, \R^N\longrightarrow \mathcal{H},\,\,T_N\left( a_0,\dots,a_{N-1} \right)=\sum_{n=0}^{N-1} a_n e_n.\]
Consider the two processes $X_{g}^{K_0,*}(t)=\sum_nX_{n,g}^{K_0,*}(t)e_n$, $K_{g,\infty}^{K_0,*}(t)=\sum_nK_{n,g,\infty}^{K_0,*}e_n$, where $X_{n,g}^{K_0,*}(t)$ and $K_{n,g,\infty}^{K_0,*}$ are defined, respectively, in \eqref{eq:X_n} and \eqref{coefficient_K_infty}. By Remark \ref{rem:sup_exp_Brownian}, there exist $C(\omega),C'(\omega)>0$ such that
\[\sum_n|X_{n,g}^{K_0,*}(t,\omega)|^2\leq C(\omega)\left(\norm{K(0)}_{\mathcal{H}} +\sum_n\frac{1}{\lambda_n}\right)<\infty ,\]
\[ \sum_n|K_{n,g,\infty}^{K_0,*}(\omega)|^2\leq C'(\omega)\sum_n\frac{1}{\lambda_n}<\infty;\]
then $X_{g}^{K_0,*}(t,\omega),K_{g,\infty}^{K_0,*}\in \mathcal{H}$.
Our first aim is to prove that the projection onto $H_N$ of $X_{g}^{K_0,*}(t)$ converges in probability. Given $\epsilon>0$, we want to prove that
\begin{equation}\label{eq:conv_prob_projection}
\lim_{t\to \infty}\PP\left(\norm{P_NX_{g}^{K_0,*}(t)-P_NX_{g,\infty}^{K_0,*}}_{\mathcal{H}}>\epsilon\right)=0.
\end{equation}
Since 
\[\norm{P_NX_{g}^{K_0,*}(t)-P_NX_{g,\infty}^{K_0,*}}_{\mathcal{H}}\leq \sum_{n=0}^N\left|X_{n,g}^{K_0,*}(t)-K_{n,g,\infty}^{K_0,*}\right|,\]
then
\begin{multline*}
\PP\left(\norm{P_NX_{g}^{K_0,*}(t)-P_NX_{g,\infty}^{K_0,*}}_{\mathcal{H}}>\epsilon\right)\leq \PP\left( \sum_{n=0}^N\left|X_{n,g}^{K_0,*}(t)-K_{n,g,\infty}^{K_0,*}\right|>\epsilon\right) \\ \leq \sum_{n=0}^N\PP\left( \left|X_{n,g}^{K_0,*}(t)-K_{n,g,\infty}^{K_0,*}\right|>\frac{\epsilon}{N}\right).
\end{multline*}
From Lemma \ref{lem:ltb_components}, $X_{n,g}^{K_0,*}(t)$ converge a.s., then in probability. Thus the right-hand side converges to zero and \eqref{eq:conv_prob_projection} is proved. Convergence in probability implies convergence in law. Since $\Pi_NK_{g}^{K_0,*}(t)$ and $\Pi_NX_{g}^{K_0,*}(t)$ are respectively functions of
 \[{\bf \beta^t}=\left(\beta_0(t)-\beta_0(t-r),\dots,\beta_N(t)-\beta_N(t-r)\right)\] and \[{\bf \beta}=\left(\beta_0(r),\dots,\beta_N(r)\right)\]
for $r\in[0,t]$, then $\mathfrak{L}(\Pi_NK_{g}^{K_0,*}(t))=\mathfrak{L}(\Pi_NX_{g}^{K_0,*}(t))$. Moreover, since $P_N=T_N\circ \Pi_N$, also $P_NX_{g}^{K_0,*}(t)$ and $P_NK_{g}^{K_0,*}(t)$ are equal in law. Then,
\begin{equation}\label{eq:conv_law_projection}
\mathfrak{L}(P_NK_{g}^{K_0,*}(t))\to \mathfrak{L}(P_NK_{g,\infty}^{K_0,*})\quad \text{for}\,\,t\to \infty.
\end{equation}
For all $\phi\in C_b(\mathcal{H})$,
\begin{align}
&\left|\int \phi(x)\mathfrak{L}(K^{K_0,*}_{g}(t))(dx)\int \phi(x)\mathfrak{L}(K^{K_0,*}_{g,\infty})(dx)\right|\\
&\leq\left|\int \phi(x)\mathfrak{L}(K^{K_0,*}_{g}(t))(dx)-\int \phi(x)\mathfrak{L}(P_NK^{K_0,*}_{g}(t))(dx)\right|\nonumber\\
&+\left|\int \phi(x)\mathfrak{L}(P_NK^{K_0,*}_{g}(t))(dx)-\int \phi(x)\mathfrak{L}(P_NK^{K_0,*}_{g,\infty})(dx)\right| \nonumber\\
&+\left|\int \phi(x)\mathfrak{L}(P_NK^{K_0,*}_{g,\infty})(dx)-\int \phi(x)\mathfrak{L}(K^{K_0,*}_{g,\infty})(dx)\right|.
\end{align}
From \eqref{eq:conv_law_projection} the second term on the right-hand side converges to zero for $t\to \infty$. Given $\eta>0$, there exists $t^*$ such that for all $t\geq t^*$,
\begin{equation*}
\left|\int \phi(x)\mathfrak{L}(P_NK^{K_0,*}_{g}(t))(dx)-\int \phi(x)\mathfrak{L}(P_NK^{K_0,*}_{g,\infty})(dx)\right|\leq \frac{\eta}{3}.
\end{equation*}
 It is easy to see that $\mathfrak{L}(P_NK^{K_0,*}_{g,\infty})(dx)=\left(\mathfrak{L}(K^{K_0,*}_{g,\infty})\circ \Pi^{-1}_N\right)(dx)$ and thus
\begin{multline*}
\int \phi(x)\mathfrak{L}(P_NK^{K_0,*}_{g}(t))(dx)=\int \phi(x)\left(\mathfrak{L}(K^{K_0,*}_{g,\infty})\circ \Pi^{-1}_N\right)(dx)
\\=\int \phi(P_N(x))\mathfrak{L}(K^{K_0,*}_{g,\infty})(dx).
\end{multline*}
We rewrite the first term on the right-hand side as
\begin{equation*}
\int \left(\phi(x)-\phi(P_N(x))\right)\mathfrak{L}(K^{K_0,*}_{g}(t))(dx). 
\end{equation*}
By the dominated convergence theorem we conclude that this term converges to zero. Then, there exists $N_1$ such that for all $N\geq N_1$,
\begin{equation*}
\left| \int \left(\phi(x)-\phi(P_N(x))\right)\mathfrak{L}(K^{K_0,*}_{g}(t))(dx)\right|\leq \frac{\eta}{3}.
\end{equation*}
The same argument is applied for the third term, and thus the convergence in law is proved.
\end{proof}

\begin{remark}
We briefly compare the result obtained above with the one obtained in the deterministic setting. If we consider $\alpha_0=0$ in $K_{g,\infty}^{K_0,*}$ , then we get exactly the limit of the detrended optimal path in the deterministic case; see Proposition 5.7 in \cite{BFFGJOEG19}. Hence, our result can be seen as a stochastic generalization of the deterministic result. 
\end{remark}

\subsection{Convergence in probability of the detrended optimal path to 0}\label{subsec:ltb_op} 

The aim of this section is to exploit the asymptotic distribution of the optimal path, $K^{K_0,*}(t)$ with $K_0\in \mathcal{H}_{e_0}^{++}$.
We will prove that in the space $\mathcal{H}$, there exists one invariant measure, a Dirac mass centered in the null process $\delta_0$.
As in the previous section, we start our investigation from the study of the Fourier components of $K^{K_0,*}(t)$.

\begin{lemma}\label{lem:ltb_components_nd}
Let Assumptions {\rm \ref{ass:state}--\ref{ass:cost}} hold. Moreover, we assume that $\lambda_1<g$, $\alpha_n(K)=\alpha_0$ for each $n\geq 0$ and $g=\frac{\lambda_0-\rho}{\sigma}
-\frac12 \alpha_0^2(2-\sigma)<0$.
Then
\[
\PP\left(\lim_{t\to\infty} K_{n,g}^{K_0,*}(t)=0\quad \forall n\geq0\right)=1.
\]
\end{lemma}

\begin{proof}
We write the Fourier expansion for $K^{K_0,*}(t)$, i.e.,
$K^{K_0,*}(t)=\sum_n K^{K_0,*}_{n}(t)e_n $,
where $ K^{K_0,*}_{n}(t)=\langle K^{K_0,*}(t),e_n\rangle$.
First, we analyze the components $ K^{K_0,*}_{n}$.
By plugging $\phi=e_n$ into the weak formulation for $K^{K_0,*}(t)$, we get
\begin{multline*}
\langle K^{K_0,*}(t),e_n\rangle=\langle K(0),e_n\rangle+\int_0^t\lambda_n\langle K^{K_0,*}(s),e_n\rangle ds\\
-\int_0^t \langle K^{K_0,*}(s),e_0\rangle\bigg\langle\left(\gamma \frac{N e_0}{f}\right)^{-1/\sigma}N,e_n\bigg\rangle ds
+\int_0^t\alpha_0\langle K^{K_0,*}(t),e_n\rangle d\beta_n(s).
\end{multline*}
Then $K_n^{K_0,*}$ is solution of the equation
\begin{equation}\label{eq:K_n_nd}
dK_n(t)=\lambda_nK_n(t)dt-c_nX_0(t)dt +\alpha_0K_n(t)d\beta_n(t),\quad K_n(0)=\langle K(0),e_n\rangle,
\end{equation}
with $c_n:=\langle(\gamma e_0)^{-1/\sigma}f^{\frac{1}{\sigma}},N^{-\frac{1-\sigma}{\sigma}}e_n\rangle $. From equality \eqref{eq:gamma} and from the definition of~$g$,
\[\lambda_0-\gamma^{-1/\sigma}\langle f^{\frac{1}{\sigma}},(Ne_0)^{-\frac{1-\sigma}{\sigma}}\rangle=g+\frac{\alpha_0^2}{2}.\]
Then for $n=0$,
\begin{equation}\label{eq:X_0_nd}
dK_0(t)=\left(g+\frac{\alpha_0^2}{2}\right)K_0dt+\alpha_0 K_0d\beta_0(t),\quad K_0(0)=\langle K(0),e_0\rangle.
\end{equation}
Thus $K_0$ is a geometric Brownian motion,
\begin{equation}\label{eq:K0_nd}
K_0(t)=K_0(0)\exp[gt+\alpha_0\beta_0(t)].
\end{equation}
By the law of the iterated logarithm, the $t$ term in the exponential will dominate, and since $g<0$, we get that $K_0(t)\to 0$ a.s.

In order to study the dynamic for $n\geq0$, we follow the same strategy as in Lemma \ref{lem:ltb_components}. First we study the dynamic for $Y_n=Z_0(t)K_n(t)$ with $Z_0(t)=K_0(t)^{-1}$, and we get that $Y_n(t)$ is solution of a linear SDE with inhomogeneous constant coefficients, i.e.,
\[dY_n(t)=\left(-g+\lambda_n+\frac{\alpha_0^2}{2}\right)Y_n(t)dt-c_ndt -\alpha_0Y_n(t) d\beta_0(t)+\alpha_nY_n(t)d\beta_n(t).\]
Then, we can write an explicit formula, 
\[Y_n(t)=\Phi^n_t\left(Y(0)+c_n\int_0^t (\Phi^n_s)^{-1}ds \right)\]
with $\Phi^n_t=\exp\big[\big(-g+\lambda_n-\frac{\alpha_0^2}{2}\big)t-\alpha_0\beta_0(t)+\alpha_0\beta_n(t)\big]$.
We can now derive the equation for $K^{K_0,*}_n$,
\begin{equation}\label{eq:Kn_nd_explicit}
K_n(t)=\tilde{\Phi}^n_t\left(K_n(0)+c_n\int_0^\infty (\tilde{\Phi}^n_s)^{-1}K_0(s)\mathds{1}_{[0,t]}(s)ds \right)
\end{equation}
with $\tilde{\Phi}^n_t=\exp\big[\big(\lambda_n-\frac{\alpha_0^2}{2}\big)t+\alpha_0\beta_n(t)\big]$.
Since  $\lambda_1<g$, then $\lambda_n\leq g<0$ for each $n\geq 1$. By the law of the iterated logarithm, the $t$ term in the exponential will dominate, and since $g<0$, we get that $K_n(0)\tilde{\Phi}^n_t\to 0$ a.s.

By following the same strategy used in Lemma \ref{lem:ltb_components} to prove \eqref{eq:limit_exp_int}, we can prove that there exists $\tilde{\Omega}$, a set of probability one, such that for each $\omega \in \tilde{\Omega}$, 

\[\lim_{t\to\infty}\int_0^\infty \tilde{\Phi}^n_t(\tilde{\Phi}^n_s)^{-1}(\omega)K_0(s,\omega)\mathds{1}_{[0,t]}(s)ds=0.\]
In conclusion, we have proved that
\[
\PP\left(\lim_{t\to\infty} K_{n,g}^{K_0,*}(t)=0\right)=1\quad \forall n\geq0.
\]
Since we are dealing with a countable set, we can write that
\begin{align*}
\PP\left(\lim_{t\to\infty} K_{n,g}^{K_0,*}(t)=0\quad \forall n\geq0\right)=1.
\end{align*}
\end{proof}

\begin{theorem}
\label{teo:convto0}
Let Assumptions {\rm \ref{ass:state}--\ref{ass:cost}} hold.
Moreover, we assume that $\lambda_1<g$, $\alpha_n(K)=\alpha_0$ for each $n\geq 0$ and that $g=\frac{\lambda_0-\rho}{\sigma}
-\frac12 \alpha_0^2(2-\sigma)<0$.
Then for each $\epsilon>0$,
 \begin{equation}\label{eq:conv_prob_K}
 \lim_{t\to\infty} \PP\left(\norm{K^{K_0,*}(t)}_{\mathcal{H}}>\epsilon\right)=0.
 \end{equation}
\end{theorem}

\begin{proof}
The proof is divided into two steps. First we prove that there exists a subset of $\Omega$, arbitrarily big where $K^{K_0,*}(t)$ convergence almost surely holds. Then, we prove that from the previous fact, one can derive convergence in probability.

{\it Step} 1. First we prove that for each $\eta\geq 0$ there exists $\Omega_\eta$ such that
\begin{equation}\label{eq:measure_Omega_eta}
\PP\left(\Omega_\eta\right)\geq1-\eta,
\end{equation}
and for each $\omega\in\Omega_\eta$,
\begin{equation}\label{eq:conv_as_Omega_eta}
\lim_{t\to\infty}\norm{K^{K_0,*}(t,\omega)}_{\mathcal{H}}=0.
\end{equation}
We write $K^{K_0,*}(t,\omega)$ in terms of a Fourier expansion, and then by recalling the explicit formula for $K_n(t)$ (see \eqref{eq:K0_nd} and \eqref{eq:Kn_nd_explicit}), we estimate the $\mathcal{H}$-norm with
\begin{multline}\label{eq:inequality_sum}
\norm{K^{K_0,*}(t,\omega)}_{\mathcal{H}}\leq \sum_{n\geq 0} \abs{K_n(t,\omega)}\norm{e_n}_{\mathcal{H}}
=K_0(t,\omega)+\sum_{n\geq 1}K_n(0)\tilde{\Phi}^n_t(\omega)\\
+\sum_{n\geq 1}\tilde{\Phi}^n_tc_n\int_0^t (\tilde{\Phi}^n_s)^{-1}K_0(s)ds.
\end{multline}
From Lemma \ref{lem:ltb_components}, there exists $\Omega_0$, whose measure is $\PP(\Omega_0)=1$ such that for each $\omega\in \Omega_0$,
\begin{equation}\label{eq:conv_first_term}
\lim_{t\to\infty}K_0(t,\omega)=0.
\end{equation}
In order to prove that also the second and  third terms of \eqref{eq:inequality_sum} converge to zero, the dominated convergence of series criterion is used. We start with the second term of \eqref{eq:inequality_sum}. First, given a positive quantity $a>0$, we split the exponential function and then we estimate by Young's inequality:
\begin{multline}\label{eq:first_serie}
\sum_{n\geq 1}K_n(0)\tilde{\Phi}^n_t(\omega)=\sum_{n\geq 1}K_n(0)\exp\left[\left(\lambda_n-\frac{\alpha_0^2}{2}\right)t+\alpha_0\beta_n(t,\omega)\right]\\
\leq\sum_{n\geq 1}K_n(0)\exp\left[(\lambda_n+a)t\right]\exp\left[-\left(\frac{\alpha_0^2}{2}+a\right)t+\alpha_0\beta_n(t,\omega)\right]\\
\leq\sum_{n\geq 1}K_n(0)^2+\sum_{n\geq 1}\exp\left[2(\lambda_n+a)t\right]\exp\left[-(\alpha_0^2+2a)t+2\alpha_0\beta_n(t,\omega)\right].
\end{multline}
Given $n\geq 1$, we apply Remark \ref{rem:sup_exp_Brownian} with $\mu=\alpha_0^2+2a$, $\sigma=2\alpha_0$, and $x_n=\exp[n]$, and we get that there exists a set $\Omega^{\eta,1}_n$ such that, for $\lambda(a)=\frac{\alpha_0^2+2a}{4\alpha_0^2}$,
\[\PP\left(\Omega^{\eta,1}_n\right)\geq 1-e^{-\lambda(a) n},\]
where $a>0$ is chosen in order to get $\sum_{n\geq 1}e^{-\lambda(a) n}$ arbitrarily small, in particular, smaller than $\eta$.
For each $\omega\in \Omega^{\eta,1}_n$,
\[\exp\left[-(\alpha_0^2+2a)t+2\alpha_0\beta_n(t,\omega)\right]\leq \exp[n].\]
We define $\Omega^\eta_1=\bigcap_{n\geq 1}\left(\Omega^{\eta,1}_n\right)$ whose measure is
\[\PP\left(\Omega^\eta_1\right)=1-\PP\left(\bigcup_{n\geq 1}\left(\Omega^{\eta,1}_n\right)^c\right)\geq 1-\sum_{n\geq 1}\PP\left(\left(\Omega^{\eta,1}_n\right)^c\right)\geq1-\sum_{n\geq 1}e^{-\lambda(a)n}\geq 1-\eta.\]
If we choose $\omega\in\Omega^\eta_1$, we get that \eqref{eq:first_serie} can be bounded by
\[\leq\sum_{n\geq 1}K_n(0)^2+\sum_{n\geq 1}\exp\left[2(\lambda_n+a)+n\right].
\]
Since $\lambda_n\sim -n^2$ and $K(0)\in \mathcal{H}$, for each $\omega\in \Omega^\eta_1$ both series are finite. Moreover, by the law of the iterated logarithm we have that $\tilde{\Phi}^n_t(\omega)$ converges to zero, almost surely; then by dominated convergence of series, we conclude that for each $\omega\in\Omega^\eta_1$,

\begin{equation}\label{eq:conv_second_term}
\lim_{t\to\infty}\sum_{n\geq 1}K_n(0)\tilde{\Phi}^n_t(\omega)=0.
\end{equation}
Now we focus on the third series of \eqref{eq:inequality_sum}. From Lemma \ref{lem:ltb_components}  we know that $K_0(t)$ converges almost surely to 0, so for each $\omega$, we can bound $K_0(t,\omega)$ with a constant depending on $\omega$, $C(\omega)$. Moreover, given a positive constant $a>0$, we can split the exponential function as
\begin{multline}\label{eq:second_serie}
\sum_{n\geq 1}\tilde{\Phi}^n_t c_n\int_0^t (\tilde{\Phi}^n_s)^{-1}K_0(s)ds
\\
=\sum_{n\geq 1}c_n\int_0^t \exp\left[\left(\lambda_n-\frac{\alpha_0^2}{2}\right)r+\alpha_0\left(\beta_n(t,\omega)-\beta_n(t-r,\omega)\right)\right] K_0(s)ds\\
\leq\sum_{n\geq 1}C(\omega)c_n\int_0^t \exp[(\lambda_n+a)r]\exp\left[-\left(\frac{\alpha_0^2}{2}+a\right)r+\alpha_0\left(\beta_n(t,\omega)-\beta_n(t-r,\omega)\right)\right]ds.
\end{multline}
Given $n\geq 1$, we apply Remark \ref{rem:sup_exp_Brownian} with $\mu=\big(\frac{\alpha_0^2}{2}+a\big)$, $\sigma=\alpha_0$, and $x_n=(1+n^\alpha)$, and we get that there exists a set $\Omega^{\eta,2}_n$ such that, for $\lambda(a)=\frac{\alpha_0^2+2a}{\alpha_0^2}$,
\[\PP\left(\Omega^{\eta,1}_n\right)\geq 1-e^{-\lambda(a) \log(1+n^\alpha)},\]
where $a>0$ is chosen in order to get $\sum_{n\geq 1}\frac{1}{(1+n^\alpha)^{\lambda(a)}}$ arbitrarily small, in particular, smaller than $\eta$, and $\alpha$ is such that $\frac{1}{\lambda(a)}<\alpha<1$. The bound from below is needed in order to have that the series $\sum_{n\geq 1}\frac{1}{(1+n^\alpha)^{\lambda(a)}}$ converges. The bound from above is used in the next calculation.
For each $\omega\in \Omega^{\eta,2}_n$,
\[\exp\left[-\left(\frac{\alpha_0^2}{2}+a\right)(t-s)+\alpha_0\left(\beta_n(t)-\beta_n(s)\right)\right]\leq (1+n^\alpha).\]
We define $\Omega^\eta_2=\bigcap_{n\geq 1}\left(\Omega^{\eta,2}_n\right)$ whose measure is
\[\PP\left(\Omega^\eta_2\right)=1-\PP\left(\bigcup_{n\geq 1}\left(\Omega^{\eta,2}_n\right)^c\right)\geq 1-\sum_{n\geq 1}\PP\left(\left(\Omega^{\eta,2}_n\right)^c\right)\geq1-\sum_{n\geq 1}\frac{1}{(1+n^\alpha)^{\lambda(a)}}\geq 1-\eta.\]
Notice that $c_n$ is uniformly bounded with respect to $n$. By the H\"older inequality, since $N\in L^\infty(S^1)$ and by assumption \eqref{condition_6},
\begin{multline*}
c_n=\int_{S_1}\left(\frac{f}{\gamma e_0}\right)^{1/\sigma}N^{-\frac{1-\sigma}{\sigma}}e_ndx\leq\norm{e_N}_{L^2}\int_{S_1}\left(\frac{f}{\gamma e_0}\right)^{2/\sigma}N^{-\frac{2(1-\sigma)}{\sigma}}dx\\
\leq\norm{e_N}_{L^2}\int_{S_1}\left(\frac{f}{\gamma e_0 N}\right)^{2/\sigma}N^{2}dx\leq\norm{e_N}_{L^2}\norm{N}_{L^\infty}\int_{S_1}\left(\frac{f}{\gamma e_0 N}\right)^{2/\sigma}dx\leq C.
\end{multline*}
If we choose $\omega\in\Omega^\eta_1$, we get that \eqref{eq:second_serie} can be bounded by
\begin{multline*}
\leq\sum_{n\geq 1}C(\omega)c_n(1+n^\alpha)\int_0^t \exp[(\lambda_n+a)(t-s)] ds\\
=\sum_{n\geq 1}C(\omega)c_n(1+n^\alpha)\frac{\left(e^{(\lambda_n+a)t}-1\right)}{\lambda_n+a}\\
\leq\sum_{n\geq 1}C(\omega)C(1+n^\alpha)\frac{e^{(\lambda_n+a)t}}{\lambda_n+a}-\sum_{n\geq 1}C(\omega)C\frac{1+n^\alpha}{\lambda_n+a}.
\end{multline*}
If we choose $\alpha<1$, since $\lambda_n\sim -n^2$, for each $\omega\in \Omega^\eta_2$ both series are finite. Then, by dominated convergence for series, we conclude that for each $\omega\in\Omega^\eta_2$,
\begin{equation}\label{eq:conv_third_term}
\lim_{t\to\infty}\sum_{n\geq 1}\tilde{\Phi}^n_t c_n\int_0^t (\tilde{\Phi}^n_s)^{-1}K_0(s)ds=0.
\end{equation}

In summary, we define $\Omega_\eta=\Omega_0\cap\Omega^\eta_1\cap\Omega^\eta_1$ whose measure is
\[\PP(\Omega_\eta)=1-\PP\left(\Omega_0^c\cup\left(\Omega^\eta_1\right)^c\cap\left(\Omega^\eta_1\right)^c\right)\geq 1-\PP(\left(\Omega^\eta_1\right)^c)-\PP\left(\left(\Omega^\eta_2\right)^c\right)\geq 1-2\eta.\]
Combining  \eqref{eq:conv_first_term}, \eqref{eq:conv_second_term}, and \eqref{eq:conv_third_term}, the first step is proved.

{\it Step} 2. Now, we prove that the previous step implies convergence in probability of $K^{K_0,*}$ to $0$, namely \eqref{eq:conv_prob_K}. Given $\eta>0$, there exists $\Omega_\eta$ whose measure is arbitrarily big, in the sense of \eqref{eq:measure_Omega_eta} and where convergence almost surely holds; see \eqref{eq:conv_as_Omega_eta}.
We use $\Omega_\eta$ to split $P(||K^{K_0,*}(t)||_H>\epsilon)$ into two terms:
\[
\PP\left(\norm{K^{K_0,*}(t)}_{\mathcal{H}}>\epsilon\right)=\PP\left(\left(\norm{K^{K_0,*}(t)}_{\mathcal{H}}>\epsilon\right)\cap \Omega_\eta\right)
+\PP\left(\left(\norm{K^{K_0,*}(t)}_{\mathcal{H}}>\epsilon\right)\cap \Omega^c_\eta\right).
\]
Given the probability space $(\Omega, \mathcal{F})$ and the subset $\Omega_\eta$, we introduce a new probability measure, \[\tilde{\PP}_\eta\left(\cdot\right)=\frac{\PP\left(\cdot\cap \Omega_\eta\right)}{\PP\left(\Omega_\eta\right)}.\]
We can now rephrase convergence \eqref{eq:conv_as_Omega_eta} in terms of the probability $\tilde{\PP}_\eta$, and we say that \eqref{eq:conv_as_Omega_eta} is equivalent to
\[\lim_{t\to\infty}\norm{K^{K_0,*}}_{L^2(S^2)}=0,\quad \tilde{\PP}_\eta\text{-a.s.}\]
Since convergence almost surely implies convergence in probability, we get that for each $\epsilon$
 \[\lim_{t\to\infty} \tilde{\PP}_\eta\left(\norm{K^{K_0,*}(t)}_{\mathcal{H}}>\epsilon\right)=0.\]
We can rewrite the first term on the right-hand side in terms of the probability measure $\tilde{\PP}$, and given the convergence in probability we can say that there exists $t^*_\eta$ such that for all $t>t^*_\eta$
\begin{align*}
\PP\left(\left(\norm{K^{K_0,*}(t)}_{\mathcal{H}}>\epsilon\right)\cap \Omega_\eta\right)=\tilde{\PP}_\eta\left( \norm{K^{K_0,*}(t)}_{\mathcal{H}}>\epsilon\right)\PP(\Omega_\eta)\leq \eta.
\end{align*}
Since $\PP\left(\Omega^c_\eta\right)=\eta$, the second term on the right-hand side can be bounded with $\eta$; thus we conclude that for all $t>t^*_\eta$,
\begin{align*}
\PP\left(\norm{K^{K_0,*}(t)}_{\mathcal{H}}>\epsilon\right)\leq 2\eta.
\end{align*}
\end{proof}

\begin{remark}
We recall that the detrending rate in the deterministic setting is $g_{det}=\frac{\lambda_0-\rho}{\sigma}$. It is worth doing a comparison with the deterministic case. If $\sigma \in (0,2)\setminus \{1\}$, then $g<g_{det}$, which implies that there exists a set of values for the parameters $\lambda_0,\rho$ such that extinction occurs in the stochastic setting and not in the deterministic one. If $\sigma\in (1,\infty)$, then $g_{det}<g$ and the situation is reversed. 
\end{remark}


\appendix
\section{Appendix}
\label{app:A} 
\subsection{Well posedness of the state equation}\label{appSE}
The rest of the section is devoted to the proof of existence of solution of \eqref{eq:K_abstract}. We start our analysis by studying the well posedness of the noise.

\begin{lemma}\label{lem:hilbert_schimdt_condition}
Let $B:\mathcal{H}\to \mathcal{L}:=\mathcal{L}(\mathcal{H};\mathcal{H})$ be the operator defined in \eqref{eq:B_ej}.
If Assumptions {\rm \eqref{ass:state}(i) and (ii)} hold, then
\begin{equation}\label{wp_mild}
\int_0^T\norm{e^{t\mathcal{L}}B(K)}_{\mathcal{L}_0}dt<\infty.
\end{equation}
Moreover, if $\Sigma_\alpha=\sum_j\norm{\alpha_j}^2_\infty<\infty$, then
\begin{equation}\label{wp_strong}
\int_0^T||B(K_s)||^2_{\mathcal{L}_0}ds=\int_0^T\text{Tr}(BB^*)ds<\infty.
\end{equation}
\end{lemma}

\begin{proof}
We start by proving \eqref{wp_strong}:
\begin{multline*}
||B(K_s)||^2_{\mathcal{L}_0}=\sum_j ||B(K)e_j||^2_{\mathcal{H}}=\sum_j ||\alpha_j(K)\langle K,e_j \rangle e_j||^2_{\mathcal{H}} \\
\leq\sum_j||\alpha_j||^2_{\infty} \langle K,e_j \rangle^2||e_j||^2_{\mathcal{H}}
\leq\norm{K}^2_{\mathcal{H}}\sum_j||\alpha_j||^2_{\infty}<\infty.
\end{multline*}
In this way, \eqref{wp_strong} is proved and we get that in the case of strong solution, the noise is well posed.
Regarding the proof of \eqref{wp_mild}, we recall that
\[
    e^{t\mathcal{L}}:\mathcal{H}\to \mathcal{H},\quad e^{t\mathcal{L}}u:=\sum_j^{\infty}e^{-t\lambda_j}\langle u,e_j\rangle e_j.
\]
By the definition of norm in the space of the Hilbert--Schmidt operator, we get
\[
\norm{e^{t\mathcal{L}}B(K)}_{\mathcal{L}_0}=\sum_j\norm{e^{t\mathcal{L}}B(K) e_j}^2_{\mathcal{H}}.
\]
From the definition of the operator $B$ (see \eqref{eq:B_ej}),
\begin{multline}\label{eq:operator}
    e^{t\mathcal{L}}B(K)e_j=e^{t\mathcal{L}}\alpha_j(K)\langle K,e_j \rangle e_j=\sum_i e^{-t\lambda_i}\langle\alpha_j(K)\langle K,e_j \rangle e_j,e_i\rangle e_i\\=e^{-t\lambda_j}\alpha_j(K)\langle K,e_j \rangle  e_j.
    \end{multline}
Plugging \eqref{eq:operator} into the definition of the Hilbert--Schmidt norm,
\begin{multline*}
 \norm{e^{t\mathcal{L}}B(K)}_{\mathcal{L}_0}=\sum_j\norm{e^{t\mathcal{L}}B(K) e_j}^2_{\mathcal{H}}=\sum_j\norm{e^{-t\lambda_j}\alpha_j(K)\langle K,e_j \rangle  e_j}^2_{\mathcal{H}}\\
 =\sum_je^{-2\lambda_jt}\langle K,e_j \rangle^2\norm{\alpha_j}^2_\infty\leq \norm{K_t}^2_{\mathcal{H}}\sum_je^{-2\lambda_jt}\norm{\alpha_j}^2_\infty.
\end{multline*}
Then, integrating in time, by assumption of boundedness on $\alpha_j$, \eqref{condition_0},
\begin{align*}
 \int_0^T\norm{e^{t\mathcal{L}}B(K_t)}_{\mathcal{L}_0}dt
 &= \norm{K_t}^2_{L^\infty([0,T];\mathcal{H})}\int_0^T\sum_je^{-2\lambda_jt}\norm{\alpha_j}^2_\infty dt\\
 &\leq C\norm{K_t}^2_{L^\infty([0,T];\mathcal{H})}\int_0^T\sum_je^{-2\lambda_jt} dt\\
& \leq C\norm{K_t}^2_{L^\infty([0,T];\mathcal{H})}\sum_j\frac{1-e^{-2T\lambda_j}}{2\lambda_j}\\
&\leq C\norm{K_t}^2_{L^\infty([0,T];\mathcal{H})}\sum_j\frac{1}{2\lambda_j}<\infty.
\end{align*}
\end{proof}

\begin{remark}
Well posedness of the noise when $\alpha_j$ are constant, namely, when directions have the same weight in the noise, is guaranteed in the setting of a mild formulation. The situation changes in higher dimension, $d=2$, where the series of the reciprocals of the eigenvalues is no longer finite and we cannot conclude positively as in the one-dimensional case.
Since from a modeling point of view there is no reason to choose some favorite direction in the noise, we will work with a mild solution. 
\end{remark}

\begin{lemma}\label{lem:lipschtzianity_B}
Let Assumptions {\rm \eqref{ass:state}(i) {\it and} (ii)} hold. The operator $B:\mathcal{H}\to \mathcal{L}:=\mathcal{L}(\mathcal{H};\mathcal{H})$ defined in \eqref{eq:B_ej} is Lipschitz, i.e., there exists a constant C such that for all $k,h\in \mathcal{H}$,
\begin{equation}\label{eq:lipschtiz_B}
\norm{e^{t\mathcal{L}}B(k)-e^{t\mathcal{L}}B(h)}_{\mathcal{L}_0}\leq C\norm{k-h}_{\mathcal{H}},\quad \norm{e^{t\mathcal{L}}B(k)}_{\mathcal{L}_0}\leq C(1+\norm{k}_{\mathcal{H}}),
\end{equation}
where C depends on $\Sigma$, $\Sigma'$, $\norm{h}_{\mathcal{H}}$.
\end{lemma}

\begin{proof}
By definition of the operator norm between normed vector spaces, 
\begin{align*}
   & \norm{e^{t\mathcal{L}}B(k)-e^{t\mathcal{L}}B(h)}^2_{\mathcal{L}_0}
    \\
    &=\sum_j\norm{e^{t\mathcal{L}}[\alpha_j(k)\langle k,e_j\rangle e_j]-e^{t\mathcal{L}}[\alpha_j(h)\langle h,e_j\rangle e_j]}^2_{\mathcal{H}}\\
      &  =\sum_j\norm{\sum_i e^{-\lambda_it}\left(\langle\alpha_j(k)\langle k,e_j\rangle e_j,e_i\rangle e_i-\langle\alpha_j(h)\langle h,e_j\rangle e_j,e_i\rangle e_i\right)}^2_{\mathcal{H}}\\
    &\leq  \sum_je^{-2\lambda_jt}\norm{\alpha_j(k)\langle k,e_j\rangle e_j-\alpha_j(k)\langle h,e_j\rangle e_j}^2_{\mathcal{H}}\\
&\hspace{3cm}-\sum_je^{-2\lambda_jt}\norm{\alpha_j(k)\langle h,e_j\rangle e_j-\alpha_j(h)\langle h,e_j\rangle e_j}^2_{\mathcal{H}} \\
    &\leq \sum_je^{-2\lambda_j}\norm{\alpha_j(k)}^2_\infty\langle k-h,e_j\rangle^2\norm{ e_j}^2_{\mathcal{H}}-\sum_je^{-2\lambda_j}\langle h,e_j\rangle^2\abs{\alpha_j(k)-\alpha_j(h)}^2\\
    &\leq \Lambda\left(\alpha_\infty+D\alpha_\infty\norm{h}_{\mathcal{H}}^2\right)\norm{k-h}^2_{\mathcal{H}},
\end{align*}
where $\Lambda=\sum_j e^{-2\lambda_j}$.
 By assumptions \eqref{ass:state}(i) and (ii), we can conclude that \[ \norm{e^{t\mathcal{L}}B(k)-e^{t\mathcal{L}}B(h)}^2_{\mathcal{L}_0}\leq  C\norm{ k-h}^2_{\mathcal{H}},\]
where $C$ depends on $\Sigma_\alpha$, $D\Sigma_\alpha$, $\norm{h}_{\mathcal{H}}$.
\end{proof}

Now we give the proof of Theorem \ref{teo:existence}, which is inspired by classical results presented in \cite{fabbri2017stochastic} 
and also by \cite{Rosestolato2019}. 

\begin{proof}
The existence of solution is proved using the Banach contraction mapping principle in $\mathcal{H}_p(0,T):=\{Z:[0,T]\times \Omega\to \mathcal{H}:\,\,\sup_{t\in[0,T]}\E\left[|Z(s)|_{\mathcal{H}}^p\right]
<\infty\}$, with $p\geq 2$.
We define the map $\mathcal{K}:\mathcal{H}_p(0,T)\to \mathcal{H}_p(0,T)$ as
\[\mathcal{K}(Y)(t)=e^{t\mathcal{L}}K_0+\int_0^t e^{(t-s)\mathcal{L}}c(s)N ds
+\int_0^t e^{(t-s)\mathcal{L}}B(K(s))dW(s),\]
and then we need to prove that the expression below belongs to the space $\mathcal{H}_p(0,T)$ and then that the map $\mathcal{K}$ is a contraction for some $T_0$ small enough.
The proof is a slight modification of the proof of Theorem 1.152 in \cite{fabbri2017stochastic}. In particular, since the term depending on the control does not depend on the variable $Y$, the proof that $\mathcal{K}$ is a contraction is exactly that of Theorem 1.152 in \cite{fabbri2017stochastic}. In order to prove that $\mathcal{K}(Y)(t)\in \mathcal{H}_p(0,T)$, we just need to control the term where the control appears linearly. Since $c\in \mathcal{A}_{S,e_0}(K)$, this term can be easily bounded:
\begin{multline*}
\E\left[\norm{\int_0^t e^{(t-s)\mathcal{L}}c(s)N ds}_{\mathcal{H}}^p\right]\leq\E\left[\left(\int_0^t \norm{ e^{(t-s)\mathcal{L}}c(s)N }_{\mathcal{H}}ds\right)^p\right]\\
\leq\E\left[\left(\int_0^t \norm{ e^{(t-s)\mathcal{L}}}_{\mathcal{H}\to \mathcal{H}}\norm{c(s)N}_{\mathcal{H}} ds\right)^p\right]
\leq C\E\left[\left(\int_0^t \norm{ N}_{\infty}\norm{c(s)}_{\mathcal{H}} ds\right)^p\right]\\
\leq C.
\end{multline*}
\end{proof}

\section*{Acknowledgments}
The authors gratefully thank Professor Franco Flandoli, Professor Raouf Boucekkine, Professor Giorgio Fabbri, Professor Athanasios Yannacopoulos, and Professor Anastasios Xepapadeas for the fruitful discussions and suggestions.

\end{document}